\documentclass[12pt]{article} 
\usepackage{amssymb,amsmath,amsthm}
\usepackage[linktoc=all]{hyperref}

\usepackage{tikz}
\usetikzlibrary{arrows.meta}
\usetikzlibrary{decorations.markings}
\usetikzlibrary{calc}

\usepackage{epigraph}

\theoremstyle{plain}
\newtheorem{theorem}{Theorem}
\newtheorem{lemma}{Lemma}
\newtheorem{proposition}{Proposition}

\theoremstyle{definition}
\newtheorem{definition}{Definition}
\newtheorem{example}{Example}
\newtheorem*{remark}{Remark}

\newcommand{\RR}{\mathbb{R}}
\newcommand{\ZZ}{\mathbb{Z}}

\newcommand{\xx}{\mathbf{x}}
\newcommand{\yy}{\mathbf{y}}

\newcommand{\xee}{\mathbf{e}}

\renewcommand{\AA}{\mathcal{A}}

\newcommand{\VV}{\mathcal{V}}

\newcommand{\eps}{\varepsilon}

\newcommand{\Csum}{\mathcal{C}\mathcal{S}}

\newcommand{\Ga}{\Gamma}
\newcommand{\GraphsNN}[1]{\mathfrak{G}(#1)}
\newcommand{\GraphsN}[1]{\widetilde{\mathfrak{G}}(#1)}

\newcommand{\beq}[1]{\begin{equation}\label{#1}}
\newcommand{\eeq}{\end{equation}}

\newcommand{\ba}[1]{\begin{array}{#1}}
\newcommand{\ea}{\end{array}}

\newcommand{\Smin}[2]{S_\land^{#1}({#2})}
\newcommand{\Smax}[2]{S_\lor^{#1}({#2})}

\newcommand{\mmin}[1]{m_{\land}^{#1}}
\newcommand{\mmax}[1]{m_{\lor}^{#1}}
\newcommand{\fname}{\mathfrak{f}}

\newcommand{\Arr}{\mathrm{Arr}} % set of preferential arrangements

\newcommand{\Rpc}{\overline{\RR}_{\geq 0}}

\newcommand{\maxx}[2]{\lor(#1\mid #2)} % max over a set
\newcommand{\minx}[2]{\land(#1\mid #2)} % min over a set
\newcommand{\Rplus}{\mathbb{R}_{>0}}
\newcommand{\RplusC}{\mathbb{R}_{\geq 0}}
\newcommand{\arcs}[1]{\mathcal{A}({#1})} % set of arcs of digraph
\newcommand{\eqnod}{\leftrightarrow}
\newcommand{\conds}[1]{{#1}^c} %condensation of digraph
\newcommand{\Max}{\mathrm{Max}} %terminal set of a partial order
\newcommand{\pows}{\mathcal{P}} %power set
\newcommand{\Ps}[1]{\mathcal{P}_0(#1)} %power set excl emptyset
\newcommand{\Mp}[1]{\mathfrak{M}_{#1}}
\newcommand{\sumn}{\mathrm{sum}}
\newcommand{\sump}[1]{\mathrm{sum}_{#1}}

\title{Three steps away from Shapiro's problem:
lower bounds for 
graphic sums
with functions `max' or `min' in denominators}
\author{Sergey Sadov\footnotemark[1]\footnote{E-mail: serge.sadov@gmail.com}}
\date{}

\begin{document}
\maketitle
\thispagestyle{empty}

%\epigraph{%
%У г р ю м - Б у р ч е е в, бывый прохвост. Разрушил старый город и построил другой на новом месте.}%
%``Destroyed the old town and built a new one in a different place''.}%
%         {Mikhail Saltykov-Shchedrin, \textit{The History of a Town}}
% 

%С одного конца хитро, с другого мудреней того; а в середке ум за разум заходит. — (о машине). 
%См. ЧУДО ДИВО МУДРЕНОЕ …   В.И. Даль. Пословицы русского народа

%Пойа с.92 Ищем ранее решенные задачи, которые получаются из данной путем обобщения, специализации или аналогии.

%Do the difficult things while they are easy and do the great things while they are small. A journey of a thousand miles must begin with a single step. 
%Lao Tzu

%You have freedom when you're easy in your harness. Robert Frost, original source unknown
%[1954, at a news conference on the eve of his 80th birthday]
%https://www.brainyquote.com/authors/robert-frost-quotes

%He didn't know what to do with the sea so he just swallowed it [??]

%“Turn your obstacles into opportunities and your problems into possibilities.”
%― Roy T. Bennett, The Light in the Heart 

\begin{abstract}

Taking Shapiro's cyclic sums $\sum_{i=1}^n x_i/(x_{i+1}+x_{i+2})$ (assuming index addition mod $n$) as a starting point, we introduce a broader class of cyclic sums, called generalized Shapiro-Diananda sums, where the denominators are $p$-th order power means of the sets $\{x_{i+j_1},\dots,x_{i+j_k}\}$ with
fixed distinct integers $j_1,\dots,j_k$ and $1\leq i\leq n$. 

Generalizing further, we replace the set of arguments
of the power mean in the $i$-th denominator by an arbitrary nonempty subset of $\{1,\dots,n\}$
interpreted as the set of out-neighbors of the node number $i$ in a directed graph with $n$ nodes.
We call such sums graphic power sums since 
their structure is controlled by directed graphs.
%they are parametrized by directed graphs.

The inquiry, as in the well-researched case of Shapiro's sums, concerns the greatest lower bound of the given ``sum'' as a function of positive variables
$x_1,\dots,x_n$. 
We show that the cases of  $p=+\infty$ (max-sums)
and $p=-\infty$ (min-sums) are tractable.

For the max-sum 
associated with a given graph the g.l.b.\ is always an integer; for a strongly connected graph it equals to graph's girth.
% equal to the sum of girths of 
%``final'' 
%certain strong components of the defining graph.

For the similar min-sum, we could not relate the g.l.b.\ 
to a known combinatorial invariant;
we only give some estimates and describe 
a method for finding the g.l.b., which has factorial complexity in %the number of nodes 
$n$.

A satisfactory analytical treatment is available for the secondary minimization --- when the g.l.b.'s of min-sums for individual graphs are mininized over the class of strongly connected graphs with $n$ nodes.
The result (depending only on $n$) is found to be asymptotic to $e\ln n$.

%Introducing a secondary 
%minimization --- over all strongly connected graphs with $n$ nodes --- we find that the 
%so defined ``second-tier'' minimum
%is asymptotic to $e\ln n$.

% 
%While linear in $n$ asymptotics have been found for the minima of a number of previously studied cyclic %sums, in this case the double minimization leads to 


\medskip
MSC primary: 
26D20,  % Other analytical inequalities
secondary:
05C35. % Graph theory, Extremal problems 
%05C38 % Paths and cycles
%15A99, %Linear algebra- miscellaneous
%
%
%49-XX CALCULUS OF VARIATIONS AND OPTIMAL CONTROL;
%OPTIMIZATION [See also 34H05, 34K35,
%49Jxx Existence theories
%49J21 Optimal control problems involving relations other than differential
%equations
%49Kxx Optimality conditions
%49K21 Problems involving relations other than differential equations
%
%39-XX DIFFERENCE AND FUNCTIONAL EQUATIONS
%39A20, %: Multiplicative and other generalized difference equations, e.g. of Lyness type
%49M05, %Methods of successive approximations->Methods based on necessary conditions
%49K99 %: Necessary conditions and sufficient conditions for optimality->None of the above, but in this section
%

%
% 90-XX OPERATIONS RESEARCH, MATHEMATICAL PROGRAMMING
%90C35, %Programming involving graphs or networks
%90C27. %Combinatorial optimization

\end{abstract}

%\clearpage

%%%%%%%%%%%%%%%%%%%%%%%%%%%%%%%%
\section{Introduction}
\label{sec:intro}

In this paper we investigate lower bounds for the sums of the form
\beq{maxsum}
S_\lor(x_1,\dots,x_n)=\sum_{i=1}^n \frac{x_i}{\max\limits_{j\in\Omega_i} x_j}.
\eeq
which will be called {\em max-sums}, and for similarly looking {\em min-sums}
\beq{minsum}
S_\land(x_1,\dots,x_n)=\sum_{i=1}^n \frac{x_i}{\min\limits_{j\in\Omega_i} x_j},
\eeq
Here $x_1,\dots,x_n$ are positive real numbers and $\Omega_1,\dots,\Omega_n$
are the given nonempty subsets of the set $\{1,\dots,n\}$.

\smallskip
The problem proposed by H.S.~Shapiro in 1954
\cite{Shapiro_1954} asked to prove the cyclic inequality
$$
\frac{2}{n}\sum_{i=1}^n \frac{x_i}{x_{i+1}+x_{i+2}}\geq C
$$
with $C=1$ for all $n\geq 3$ and positive $x_1,\dots,x_n$. It is assumed that $x_{i+n}=x_{i}$.
%, $i=1,2$.
%, hence the term ``cyclic inequality''. 
The elementary case $n=3$ (Nesbitt's inequality)  appeared in print in 1903 at the latest
\cite{Nesbitt_1903}. 
Shapiro's original conjecture was shown to be generally wrong. The correct best constant, found by Drinfeld in 1971 \cite{Drinfeld_1971}, is $C\approx 0.989133$.
Details of the story about Shapiro's inequality can
be found, for instance, in \cite[Ch.~16]{Mitrinovic_1993}, \cite{Fink_1998} or \cite{Clausing_1992}. 

\smallskip
Let us describe the ``three steps'' from the headline that link Shapiro's problem
to the present one. 

Step 1 is a blunt generalization: the first summand  $x_1/(x_2+x_3)$ defining the cyclic pattern of Shapiro's sum is replaced by some function of a fixed (independent of $n$) number of variables. In such a broad setting,
statements of general nature may be available, see e.g.\ \cite{Goldberg_1960}; however, hardly anything quantitatively interesting can be proposed. 
%, which are then cyclically shifted to make up a cyclic %sum of the general form. 

Step 2 purports to be an intelligent specialization. Consider the family of patterns of the form $x_1/(|J|^{-1}\sum_{i\in J} x_i^p)^{1/p}$, where $p\in\RR$ and $J$ is some fixed set. The corresponding cyclic sums for specific sets $J$ have been studied in the literature; the problem is difficult and the results are sketchy --- 
%and incomplete 
see Sec.~\ref{sec:cycsum}.  The situation is different in the limiting cases $p=\pm\infty$, that is, for patterns of the form
$x_1/\min_{i\in J} x_i$ or $x_1/\max_{i\in J} x_i$.
Then the problem becomes very tractable:
in essense, complete results can be obtained through an application of the inequality between the arithmetic and geometric means (AGM). We treat these cases from the convenience of a prepared, more general position; see Proposition~\ref{prop:maxsum-cyclic}
for max-sums and
Proposition~\ref{prop:minsum-cyclic} for min-sums.

The purpose of Step 3 is to partly restore the degree of nontriviality by a bold structural modification of the objective function. For patterns of Step~2 with $|p|<\infty$, the nature of complications was analytical (intractability of conditions of extremum). 
%(or algebraic, if one thinks of the critical point equations). 
Now we make it combinatorial with $p=\pm\infty$, assigning to every index $i$ an arbitrary set $\Omega_i\subset\{1,\dots,n\}$ that prescribes on which of the $x_j$'s the function in the $i$-th denominator depends. In the cyclic case, $\Omega_i=(\Omega_1+i-1)\mod n$.
In general, we define the correspondence $i\mapsto \Omega_i$ in terms of directed graphs, hence the term ``graphic sums'' (Sec.~\ref{sec:grsum}). 

\smallskip
The proposed metamorphosis
%transformation 
of Shapiro's problem, which admittedly looks rather artificial, 
is motivated by the desire to obtain a setting, quite general, while amenable to analysis, yet nontrivial. It is hoped that the reader will find our choice also justified by the final results (particularly, Theorems~\ref{thm:max-girth} and \ref{thm:minminsum-strong}), 
%Sec.~\ref{sec:max} and Sec.~\ref{sec:min}), 
which, in 
%a view cherished by 
a biased view of 
this author, 
are not without elegance. These results are analytycal; there is also an open question concerning efficient computation of the greatest lower bound for a given individual min-sum (Sec.~\ref{sec:min}).

\smallskip
%Our transformation makes perhaps a little more
%sense than the brutal action of the lunatic mayor
%from the epigraph.
%
If one compares
Shapiro's problem to a small nice cut in a rough rock, long on display, then our present endeavour is alike to chipping at a different spot aiming to create a new attractive facet --- and maybe (not expecting too much) --- to get a bit closer to revealing deeper secrets of the entire crystal. 

\smallskip
Last but not least: the analysis of the functions  \eqref{maxsum} and especially \eqref{minsum} prompted the author to examine a class of extremal problems
of the type ``minimize the sum of $n$ variables provided products of certain sets of those variables are given'' --- a generalization of the AM-GM problem parametrized by combinatorial data.
This led to a little theory \cite{Sadov2020} that may present an independent interest. 
Here we make use of some results of \cite{Sadov2020}. 

\smallskip
The results in the narrow sense (more significant, called Theorems and others, called Propositions) with their proofs occupy only about half of the paper.
% are just sketched because their core can be found in
%
The author feels that it is desirable to present the subject for the first time in an appropriate, sufficiently broad context even if the generality of the primary definitions is not matched by the generality of theorems.

%\smallskip
%[Structure , formal]
%In Section~\ref{sec:motivation} we set up a formal context which indirectly relates the problem to the well-known Shapiro's cyclic inequality. 

\section{Cyclic sums of Shapiro-Diananda type}
\label{sec:cycsum}

The purpose of this section is to fix terminology and notation and to briefly review known results about sums structurally similar to Shapiro's, with power means in the denominators. Reading the review is not necessary
for understanding Section~\ref{sec:grsum} and the sequel;
however, it may help readers to get a more holistic perception of the subject.

\medskip
The cardinality of a finite set $\Omega$ is denoted $|\Omega|$.

For integers $a\leq b$, we denote by $[a,b]$ the set of $b-a+1$ consequtive integers $a,a+1,\dots,b$. If $a=1$, we use the abbreviation
$[n]=[1,n]=\{1,2,\dots,n\}$.
Reduction modulo $n$ in our context means, for the given $a$, finding $a'\in [n]$ such that $a\equiv a'\mod n$.
(So $0\!\!\mod n=n$.)

Suppose $n$ is fixed. 
The $n$-tuple $(x_1,\dots,x_n)$ will be abbreviated as $\xx$.
Let $\tau$ denote the (left) cyclic shift, $\tau:\,[n]\to[n]$, $i\mapsto i+1 \mod n$.
%, where $x_{n+i}$ is identified with $x_i$. 
Put $\xx^\tau=(x_2,\dots,x_{n},x_1)$.

If $J=(j_1,\dots,j_k)$ is an ordered $k$-tuple of integers, then 
we put $(\xx\mid J)=(x_{j_1'},\dots, x_{j_k'})$,
where $j_i'$ means $j_i$ reduced modulo $n$. 

%An interval $[a,a+1,\dots,b]$ of the set $\ZZ$ of integers
%will be abbreviated as $[a:b]$. Clearly, $|[a:b]|=b-a+1$.

%Given a function $f$ of $k=b-a+1$ variables labelled as
%$x_{a},x_{a+1},\dots,x_{b}$, we define the $n$-ary %shift-invariant 
%function, the {\em cyclic sum with local pattern $f$}, by
%$$
% \Csum_n[f](x_1,\dots,x_n)=\sum_{i=0}^{n-1} %f(\xx^{\tau^i}\mid[a:b]).
%$$

Given a function $f$ of $k=|J|$ variables, 
%labelled as $x_{j_1},\dots,x_{j_k}$, 
we define the $n$-ary shift-invariant 
function, the {\em cyclic sum with local pattern $f$}, by the formula
$$
 \Csum_n[f](x_1,\dots,x_n)=\sum_{i=0}^{n-1} f(\xx^{\tau^i}\mid J).
$$

One of the earliest examples with indefinite number of variables involving a function of such a form is Boltzmann's inequality 
\cite[Ch.~7, \S~81]{Boltzmann_1898}
$$
 \Csum_n[(x_1-x_2)\log x_1](\xx)=\log \left(x_1^{x_1-x_2}
x_2^{x_2-x_3}\cdots x_n^{x_n-x_1}\right) \geq 0.
$$

If $\varphi$ is some function of two variables
and $f(x_1,x_2,x_3)=\varphi\left(\frac{x_2}{x_1},\frac{x_3}{x_1}\right)$,
then
\beq{GodLev_sum}
 \Csum_n[f](\xx)=\varphi\left(\frac{x_2}{x_1},\frac{x_3}{x_1}\right)+\varphi\left(\frac{x_3}{x_2},\frac{x_4}{x_2}\right)+
\dots
%+\varphi\left(\frac{x_{n}}{x_{n-1}},\frac{x_1}{x_{n-1}}\right)
+\varphi\left(\frac{x_{1}}{x_{n}},\frac{x_2}{x_{n}}\right).
\eeq
Shapiro's $n$-th sum occurs when $\varphi(x,y)=\frac{2}{x+y}$.

\smallskip
Speaking about domains of the functions being discussed, it is sufficient for our purposes to always assume that all $x_i\geq 0$ and exclude those tuples that cause one or more of the denominators to vanish.

\smallskip
Put 
$$
 M_{k,p}(x_1,\dots,x_k)=\left(\frac{x_1^p+\dots+x_k^p}{k}\right)^{1/p}
$$
and, more generally, for $J=(j_1,\dots,j_k)$,
$$
 M_{n,J,p}(\xx)=M_{n,J,p}(x_{1},\dots,x_{n})=\left(\frac{1}{|J|}\sum_{i=1}^k x_{j_i'}^p\right)^{1/p}.
$$
Here, as before, $j_i'$ means $j_i$ reduced modulo $n$.

The limit cases are
$$
 M_{n,J,-\infty}(x_{j_1},\dots,x_{j_k})=\min_{j\in J} x_{j'}
$$
and
$$
 M_{n,J,+\infty}(x_{j_1},\dots,x_{j_k})=\max_{j\in J} x_{j'}.
$$
Let us call the following functions {\em the cyclic power sums of Shapiro-Diananda type} ({\em cyclic $p$-sums}, for short):
\beq{SDp_sum}
S_{n;k,p}(x_1,\dots,x_n)=\Csum_n\left[\frac{x_1}{M_{k,p}(x_2,\dots,x_{k+1})}\right](\xx)=
\left(\frac{k x_1^p}{x_2^p+\dots+x_{k+1}^p}\right)^{1/p}
+\cdots 
\eeq
and the generalized cyclic $p$-sums
\beq{SDp_gsum}
S_{n;J,p}(x_1,\dots,x_n)=\Csum_n\left[\frac{x_1}{M_{n,J,p}(x_1,\dots,x_n)}\right](\xx).
\eeq
Shapiro's $n$-th sum is $S_{n;2,1}$ or, as a particular case of \eqref{SDp_gsum},
$S_{n;(2,3),1}$.

\smallskip
It is obvious that for any $p\in[-\infty,+\infty]$
$$
S_{n;J,+\infty}(\xx)\leq S_{n;J,p}(\xx)\leq S_{n;J,-\infty}(\xx).
$$
Moreover, by the monotonicity of power means 
\cite[{\S\,16}]{HLP} it follows that
$$
p<p'\quad\Rightarrow\quad
S_{n;J,p}(\xx)\geq S_{n;J,p'}(\xx).
$$

Putting $x_1=\dots=x_n=1$, we get
\beq{triv_est_Snkp}
 \inf_{\xx} \frac{1}{n} S_{n;k,p}(\xx)\leq 1
\eeq
and, more generally,
\beq{triv_est_SnJp}
 \inf_{\xx} \frac{1}{n} S_{n;J,p}(\xx)\leq 1.
\eeq
for any finite set $J\subset \ZZ$.

Beyond these trivial obervations, minimizing cyclic sums is, for the most part, a subtle game of balancing, only patchily explored.
The big question is: when the inequality in \eqref{triv_est_Snkp} or \eqref{triv_est_SnJp} is strict --- and, in that case, what is the value of the greatest lower bound.

\smallskip
Under very weak conditions on the pattern $f$, Goldberg \cite{Goldberg_1960}
proved that
$$
 \inf_{n} \, \inf_{\xx}\frac{1}{n}\Csum_n[f](\xx)=
\lim_{n\to\infty} \, \inf_{\xx}\frac{1}{n}\Csum_n[f](\xx).
$$

Following Drinfeld's analysis \cite{Drinfeld_1971} of Shapiro's problem,
a satisfactory method to explore the asymptotics of the g.l.b.\ of sums
$\Csum_n[f](\xx)$
%best asymptotic ($n\to\infty)$ constant, following Drinfeld's analysis \cite{Drinfeld_1971} of Shapiro's problem,  
exists for pattern functions depending on $x_1$, $x_2$, $x_3$. Godunova and Levin \cite{GodLev_1973}, \cite{GodLev_1976}, \cite{GodLev_1982} applied it to a wide class of sums of the form \eqref{GodLev_sum} and to a few separate, special cases. See also \cite[\S~3.1]{Finch_2003}.

\iffalse
Godunova, Levin, MZ, 1973
%Е. К. Годунова, В. И. Левин, Уточнение
%некоторых циклических неравенств, Матем. заметки, 1973, том 14, выпуск 3, 305--316
$$
 f(x_1,x_2,x_3)=\left(\frac{a_1}{(a_2+a_3)/2}\right)^\lambda, \quad \lambda\geq 2
$$
--- ref. to Daykin[:10]

A class of functions with pattern
$$
f(x_1,x_2,x_3)=\phi(x_1/x_2, x_1/x_3).
$$
(Minimization of Daykin's sums reduces to these.)

Diananda proved Daykin's for $\lambda>(\sqrt{5}+1)/2$
(
Diananda, P. H. Some cyclic and other inequalities.
IV. Proc. Cambridge Philos. Soc. 76 (1974), 183.
)

P. H. Diananda (1962). Some Cyclic and Other Inequalities. Mathematical Proceedings of the
Cambridge Philosophical Society, 58, pp 425--427 doi:10.1017/S0305004100036653

denominator --- homo polynomial of degree 1.

%Е. К. Годунова, В. И. Левин, О точности
%нетривиальной оценки в одном циклическом неравенстве, Матем. заметки, 1976, том 20, выпуск 2, 203--205

$$
f(x_1, x_2, x_3)=\frac{a_1 x_1+a_2 x_2+a_3 x_3}{b_1x_2+b_2 x_3}
$$

%Е. К. Годунова, В. И. Левин, Нижняя оценка одной циклической суммы, Матем. заметки, 1982, том 32, выпуск 1, 3--7
$$
f(x_1, x_2, x_3)=\frac{ x_1+p x_1 x_2+q x_2^2}{x_2^2+ x_3^2}
$$
\fi
%%%%%%%%%%%%%%%%%%%%%%%%%%%%%%%%%%%%%%%%%

No method of strength comparable to Drinfeld's  --- that is, capable of providing the best lower bound for a family of suitably normalized rational cyclic sums, --- is presently known for patterns that depend on more than three variables or even for a pattern of the form $f(x_1,x_2,x_4)$. 
Even for a fixed $n$ analytical or numerical
minimization of a particular cyclic sum can be difficut; the review \cite{Clausing_1992} and the paper \cite{Ando_2013} provide good evidence.  

The sums $S_{n;k,1}$ were first considered by Diananda
\cite{Diananda_1959}. The best currently known estimate applicable to arbitrarily large values of $k$ and $n$ was obtained, using considerations similar to Drinfeld's, in the author's paper
\cite{Sadov_2016}:
$$
k(2^{1/k}-1)\leq \lim_{n\to\infty}\,\inf_{\xx}\,\frac{1}{n}S_{n;k,1}(\xx)
\leq \gamma_k,
$$ 
where $\gamma_k$ is the root of a certain transcendental equation. The sequence $(\gamma_k)$ monotonely decreases;
$\gamma_2\approx 0.98913$ is Drinfeld's constant
and $\lim_{k\to\infty}\gamma_k\approx 0.93050$. Note that
$k(2^{1/k}-1)\to\ln 2\approx 0.693$, so there is a significant gap between the lower and upper bounds.
A rigorously justified computational procedure for determination of the true value of the limit ($n\to\infty$) is not known even for $k=3$; neither it is known for the sums $S_{n;J,1}$ with $J=\{2,4\}$, very similar visually to Shapiro's. Plausible numerical results are presented in 
\cite[\S~7]{Boarder_Daykin_1973}. The mentioned ``simplest cases'' of the
%simplest cases of the 
challenge in the explicit form read as follows. 

\bigskip\noindent
{\bf Open problem.}
Find
(or estimate numerically, with justification) the limits
$$
\lim_{n\to\infty}\inf_{\xx} \frac{1}{n}\left(\frac{x_1}{x_2+x_3+x_4}+\frac{x_2}{x_3+x_4+x_5}+\dots+\frac{x_{n-1}}{x_n+x_1+x_2}+\frac{x_n}{x_1+x_2+x_3}\right)
$$
and
$$
\lim_{n\to\infty}\inf_{\xx} \frac{1}{n}\left(\frac{x_1}{x_2+x_4}+\frac{x_2}{x_3+x_5}+\dots+\frac{x_{n-1}}{x_n+x_2}+\frac{x_n}{x_1+x_3}\right).
$$
%To reiterate: the problem is (to the author's knowledge) open.

\smallskip
Daykin \cite{Daykin_1971} 
considered the function 
$$
 \Phi_k(\nu)=\inf_{\xx} S_{n;k,1/\nu} (\xx)
$$
with $k=2$ and showed that it is continuous and convex for
$\nu\in(0,+\infty)$
and that $\Phi_2(\nu)=n$ for $\nu\geq 2$.
Diananda \cite{Diananda_1973}, extending the research of Daykin, found that the functions $\Phi_k(\nu)$ with any $k\in\{2,3,\dots\}$ are continuous and convex for $\nu\in(0,+\infty)$, and that
\beq{DiaDay}
 \inf_{\nu>0} \Phi_k(\nu)=\lim_{\nu\to 0^+} \Phi_k(\nu)=
\left\lfloor\frac{n+k-1}{k}\right\rfloor.
\eeq
We will revisit this formula in Sec.~\ref{sec:max} in connection with  Theorem~\ref{thm:max-girth}, see \eqref{DiaDay2}.

Generalizing another Daykin's observation, Diananda  noted that $\Phi_k(\nu)\geq n$ when $\nu\leq 0$.
In \cite{Diananda_1974} he investigated the 
%positive 
threshold value $\nu_n=\inf\{\nu>0\mid \Phi_k(\nu)=n\}$, which bounds the region where the trivial estimate of the type \eqref{triv_est_Snkp} is valid, and proved that
$$
1.036<\sup_n \nu_n\leq\frac{\sqrt{5}+1}{2}.
$$

\smallskip
The discussion so far pertained to bounding cyclic sums from below.
In closing this section let us address the natural question: 
does the analogous {\it maximization}\ problem: to determine $\sup_{\xx} S_{n;J,p}(\xx)$ --- make sense?

\smallskip
%Suppose $p>0$.
%
If $1\notin J$, then, letting $x_1\to\infty$ while keeping $x_2,\dots,x_n$ fixed, we get
$\sup_{\xx} S_{n;J,p}(\xx)=\infty$.

\smallskip
The case $p=1$, $J=[a,b]$ with $a,b\in\ZZ$, $a\leq 1\leq b$
was studied by Baston \cite{Baston_1973}
who considered cyclic sums with pattern $M_{1,\ell}(x_1,\dots,x_\ell)/M_{1,k+\ell}(x_a,\dots,x_b)$, where $b=a+k+\ell-1$ and $a\leq 1\leq a+k$ (that is, $[1,\ell]\subset[a,b]$). 
%
%The analysis is simple but the answers are not immediately obvious.
%
Baston's result in the case $\ell=1$ reads:
%We use Baston's argument to obtain the result for an arbitrary $J$. 
%(and to correct an inaccuracy in \cite[Theorem~1]{Baston_1973} when $J=[a,\dots,b]$).
%
%\begin{proposition}

\smallskip
{\em If $a\leq 1\leq b$ and  $J=[a,b]$, then 
$$
 \frac{1}{n}\sup_{\xx}  S_{n;J,1}(\xx)\leq \frac{|J|}{1+\min(1-a,b-1)}.
$$
The equality occurs for infinitely many $n$.}

%\end{proposition}

\medskip
Beyond that, hardly anything is known about maximization of cyclic $p$-sums with $0<p<\infty$.
Yamagami \cite{Yamagami_1993} solved the maximization problem for a cyclic sum with weighted sum in the denominator of the pattern:
$$
  \max_{\xx}\Csum_n\left[\frac{x_1}{(s-m)x_1+\sum_{j\in J}x_{1+j}}\right](\xx)=\frac{n}{s},
$$
where $J=[\ell,\ell+m-1]$, $s\geq n$, $1\leq \ell\leq n-m$.
He conjectured that the result remains true for any set  $J\subset[1,n]$ of cardinality $m$ 

\smallskip
Our short review is not exhaustive; however, the author is not aware of any substantial results concerning power sums of Shapiro-Diananda type not mentioned or not referred to in the quoted references. Despite a seemingly elementary character of the questions,
the area remains wide open for systematic exploration as well as for casually trying various particular cases, possibly even as students' projects.\footnote{V.G.~Drinfeld was a 10th grade student working under supervision of Prof.~V.I~.Levin when he proved that $\lim_{n\to\infty} \inf_{\xx} n^{-1} S_{n;2,1}(\xx)=\gamma_2$.}

\section{Graphic sums}
\label{sec:grsum}

Let $\Omega$ be a finite set, $\pows(\Omega)$ its  power set, and $\Ps{\Omega}=\pows(\Omega)\setminus\{\emptyset\}$  
be the set of all nonempty subsets of $\Omega$. 

By a ``function with variable number of arguments''
we understand a functional symbol $\fname$, such as $\min$, $\max$, $\sumn$, that comprises a family of functions $\fname^{(k)}$, one for every arity $k=0,1,2,\dots$. For example, $\fname(x_1,x_2)=\fname^{(2)}(x_1,x_2)$.
%(the number  of scalar arguments). 
In the terminology of programming, $\fname$ can be viewed as a virtual function whose actual instance at the place of occurence is determined by the signature (here, the number $k$ of scalar variables).

We will consider only symmetric functions. Given 
a vector $\xx$ with index set $[n]$ and
a subset $\Omega=\{i_1,\dots,i_k\}\subset [n]$, we write 
$$
\fname(\xx|\Omega)=\fname(x_{i_1},\dots,x_{i_k}).
$$ 
This notation includes the case $\Omega=\emptyset$:
then the right-hand side is the constant $\fname^{(0)}()$ --- the instance of $\fname$ of arity zero.

\begin{remark} If $\fname$ is not necessarily symmetric,
then the above definition would make sense assuming that
$\Omega$ is an {\em ordered subset}\ of $[n]$.
\end{remark}

%\medskip\noindent
For our purposes it suffices to assume that always $\xx\geq 0$ (i.e.\ $x_1\geq 0,\dots,x_n\geq 0$) and that $\fname(\xx|\Omega)$ takes values in $\Rpc=[0,\infty]$.
Thus 
$$
(\xx,\Omega)\mapsto \fname(\xx|\Omega).
$$
is a map $\RplusC^n\times\pows([n])\to \Rpc$. The target set $\Rpc$ has natural linear order, the set  $\RplusC^n$ is partially ordered by 
coordinate-wise comparison:
$\xx\geq \tilde \xx\;\Leftrightarrow\; (\forall j\; x_j\geq \tilde  x_j)$, and the set $\pows([n])$ is partially ordered by set inclusion. Hence it makes sense to inquire whether $\fname(\cdot|\cdot)$ is monotone (order-preserving or order-reversing) 
% [cf. terminology: Stanley v.1, Sect.4.5]
separately in $\xx$ and $\Omega$.

We call $\fname$ {\it ascending}, resp., {\it descending}\ if $\fname(\xx|\cdot)$ is order-preserving, resp., order-reversing, with respect to the set argument for any fixed $\xx$.

The functions of primary interest to us will be
$\max$ (abbreviated as $\lor$) and $\min$ (abbreviated as $\land$)
with conventions
$$
\maxx{\xx}{\emptyset}=0,
\qquad
\minx{\xx}{\emptyset}=\infty.
$$
We will also consider the functions $\sump{p}$ (power sums of order $p$) and $\Mp{p}$ (power means of order $p$)
defined by the familiar expressions
$$
\sump{p}(x_1,\dots,x_k)=\left(\sum_{i=1}^k x_i^p\right)^{1/p}
$$
and
\beq{Mp_vs_sump}
 \Mp{p}(x_1,\dots,x_k)=\left(\frac{1}{k}\sum_{i=1}^k x_i^p\right)^{1/p} =k^{-1/p}\sump{p}(x_1,\dots,x_k).
\eeq
Here $p\in\RR\setminus\{0\}$. 

The $0$-ary instances are $\Mp{p}()=0$ for $p>0$ and 
$\Mp{p}()=+\infty$ for $p<0$. 

The functions $\lor$ and $\land$ are the limit cases:
$$
 \land=\sump{-\infty}=\Mp{-\infty},
\qquad
 \lor=\sump{+\infty}=\Mp{+\infty}.
$$

In the case $p=0$, which was excluded above, $\Mp{0}(x_1,\dots,x_k)$ can be naturally defined as the geometric mean $(x_1\cdot\dots\cdot x_k)^{1/k}$ with convention $\Mp{0}()=1$. The monotonicity in $p$ (implying continuity at $p=0$ if $k\geq 1$)
is thus ensured.

All these functions 
%with variable number of arguments 
($\fname=\sump{p}$, $\Mp{p}$, including $\land$ and $\lor$)
are  homogeneous of order one:
$$
\fname(\lambda\xx|\Omega)=\lambda\fname(\xx|\Omega),
\quad \forall \lambda>0.
$$
Also, they all are monotone (increasing) with respect to the $\xx$-argument:
$$
\xx\geq\xx'
\quad\Rightarrow\quad 
\fname(\xx|\Omega)\geq \fname(\xx'|\Omega).
$$
%where $\fname$ is any of 

Clearly, $\sump{p}$ is ascending if $p>0$ and descending if $p<0$. In particular, 
$\lor$ is ascending and $\land$ descending:
$$
\Omega\supset\Omega'
\quad\Rightarrow\quad 
\maxx{\xx}{\Omega}\geq \maxx{\xx}{\Omega'},
\quad
\minx{\xx}{\Omega}\leq \minx{\xx}{\Omega'}.
$$
The power means with $|p|<\infty$ are neither ascending nor descending.

\medskip
Given a symmetric function $\fname$ with variable number of arguments let us introduce the function of a vector argument $\xx\in\RplusC^n$, or equivalently, of the fixed number $n$ of nonnegative real variables, whose structure is modelled after the Shapiro-Diananda cyclic sums:
\beq{f-sum}
\xx \;\mapsto\; 
\sum_{i=1}^n \frac{x_i}{\fname(\xx|\Omega_i)}.
\eeq
Here $i\mapsto\Omega_i$ is a given assignment of nonempty subsets of $[n]$. It is natural to interpret such an assignment in terms of a directed graph (digraph).      

A {\it simple digraph}\ is a pair $\Ga=(\VV,\Ga^+)$ where $\VV=\VV(\Ga)$ is the {\it set of nodes}\ and the map $\Ga^+:\,\VV\to\pows(\VV)$ determines the {\it set of arcs}\
(directed edges) outgoing from each node.

The whole set of arcs of $\Ga$ is
$$
 \AA(\Ga)=\{(v,v')\in\VV\times\VV\mid v'\in \Ga^+(v)\}.
$$
We will also use the intuitive notation $v\to v'$ to denote the adjacency $v'\in\Ga^+(v)$.

The {\it outdegree}\ of a node $v$ is $d^+(v)=|\Ga^+(v)|$. 
%The number of vertices $n=|\VV|$ is called the {\it order}\ of $\Ga$.

Note that {\it loops} (i.e.\ arcs of the form $v\to v$) are allowed in a simple digraph, while multiple arcs are not.

In Section~\ref{sec:min} we will need digraphs that are not necessarily simple (multiple arcs are allowed).
Such a digraph is defined by specyfing its set of arcs $\AA$, the set of nodes $\VV$, and two adjacency functions $\AA \to \VV$: $\alpha$ (the beginning of arc) and $\beta$ (the end of arc). 

From now on, we write ``graph'' meaning ``simple digraph'' throughout, except in Sec.~\ref{sec:min}, where non-simple digraphs will appear in a well-defined context.

In the sum \eqref{f-sum}, we demanded the sets $\Omega_i$ to be nonempty. The reason is to avoid expressions that are nowhere finite: indeed, for $\fname=\sump{p}$ with
$p>0$, say, we have $\fname(\xx|\emptyset)=0$ and a summand in \eqref{f-sum} with
$\Omega_i=\emptyset$ would cause the sum to be infinite for any $\xx>0$. We want to avoid such a situation in the following definition; hence we introduce the notation
$\VV'(\Ga)=\{v\in\VV(\Ga)\mid d^+(v)>0\}$.

\begin{definition}
\label{def:grsums}
Let $\Ga$ be a graph and $\xx\in\RplusC^\VV(\Ga)$ be a vector
with components indexed by the nodes of $\Ga$. The {\it graphic $\fname$-sum associated with $\Ga$}\
is the function with values in $\Rpc$ defined by
\begin{equation}
\label{grsum-gen}
 S^{\Ga}_{\fname}(\xx)=\sum_{v\in\VV'(\Ga)}\frac{x_v}
{\fname(\xx|\Ga^+(v))}.
\end{equation}
The domain of $S^{\Ga}_{\fname}(\cdot)$ consists of vectors
$\xx\in\RplusC^n$ $(n=|\VV|)$,
called {\it admissible}, such that $\xx\neq 0$ and  ${\fname(\xx|\Ga^+(v))}\neq 0$ whenever $x_v=0$.
\end{definition}

We allow that $S^{\Ga}_{\fname}(\xx)=\infty\;$
for some $\xx$ but we forbid summands of the form ${0}/{0}$. 
The admissibility condition can be stated as the system of inequalities
$\max(x_v,\fname(\xx|\Ga^+(v)))>0$ for all $v\in\VV(\Ga)$.

For the functions $\fname$ considered earlier in this section,
every summand in the right-hand side of \eqref{grsum-gen}
is a homogeneous function of degree 0. Hence
$$
 S^{\Ga}_{\fname}(\lambda\xx)=S^{\Ga}_{\fname}(\xx),
\quad \forall\lambda>0.
$$

\smallskip
The introduced framework makes it possible to discuss
Shapiro-type problems from a wider perspective.

\medskip\noindent{\bf Generalized Shapiro's problem for graphic sums.}
{\em Given the graph $\Gamma$, find the greatest lower bound
\begin{equation}
\label{infSgen}
m^{\Ga}_{\fname}=\inf_{\xx} S^{\Ga}_{\fname}(\xx)
\end{equation}
over the set of admissible vectors. }

\medskip
Apart from some simple transformations (see end of this section and Sec.~\ref{sec:strong}), this problem is too broad,
and a specialization is needed to enable a meaningful treatment. 

%When $\fname=\lor$, we will use the term {\it max-sum}, omitting the implicit adjective {\it graphic}.
%Similarly, graphic sums with $\fname=\land$ will be
%called {\it min-sums}.

\medskip\noindent
{\bf Remark} concerning maximization of $S^\Ga_\fname(\xx)$.
One can similarly state a general maximization problem, replacing '$\inf$' in the right-hand side of \eqref{infSgen} by '$\sup$'. We cited Baston's result 
concerning cyclic sums at the end of section~\ref{sec:cycsum}. Similarly to the subcase $1\notin J$ in the case of cyclic sums, the general problem is trivial if $v\notin\Ga^+(v)$ for some $v\in\VV(\Ga)$: then 
%Indeed, making $x_v$ arbitrarily large, we obtain that 
$\sup_\xx S^\Ga_\fname(\xx)=+\infty$. 
If $v\in\Ga^+(v)$ for all $v$, then the problem is still trivial for $\fname=\lor$ or $\land$. 
Indeed, $x_v/\lor(\xx|\Ga^+(v))\leq 1$, so $S^\Ga_\lor(\xx)\leq |\VV(\Ga)|$. Taking $\xx=(1,1,\dots,1)$, we see that 
$\max_{\xx} S^\Ga_\lor(\xx)= |\VV(\Ga)|$.
Finally, the only case where
$\sup_\xx S^\Ga_\land(\xx)$ is finite occurs when
$\Ga^+(v)=\{v\}$ for all $v\in\VV(\Ga)$. Then, obviously,
$S^\Ga_\land(\xx)\equiv|\VV(\Ga)|$.

%One example concerning maximization of a graphic sum with a non-symmetric function $\fname$ found in the literature
%, with a weighted sum in the denominator of the pattern, 
%is Yamagami's conjecture, see end of \cite{Yamagami_1993}. There, $\fname^{(m)}(\yy)=(n-m)y_1+y_2+\dots+y_{m+1}$ (depends on $n$).

%%%%%%%%%%%%%%%%%%%
\iffalse
If for at least one $i$ we have $\emptyset\neq\Ga^+(i)\not\ni i$, then letting $x_i\to \infty$ we get
$\sup_{\xx} S^{\Ga}_{\fname}(\xx) =\infty$. Hence, in a meaningful problem every vertex of $\Ga$ must have a loop.
Even then, assuming that $\fname$ is ascending and $\inf_{\yy\in\Rplus^k}y_1^{-1}\fname(\yy)=0$ for all $k\geq 2$ (as is the case for the power sums with $p<0$), the presence of any vertex with $|\Ga^+(i)|>1$ implies $\sup_{\xx} S^{\Ga}_{\fname}(\xx) =\infty$. 

Maximization problem for {\it cyclic, interval}\ sum-sums ($\fname=\sumn$, $\Ga^+(i)=[i-a,i+b]\cap\ZZ$, $a\geq 0$, $b\geq 0$,
arithmetic modulo $n$) is covered by Baston's results \cite{Baston_1973}. The analysis is simple but the answers are not immediately obvious.
Beyond that, hardly anything is known about maximization of $\fname$-sums with $\fname=\sumn_p$, $0<p<\infty$.

For $p=\infty$, i.e.\ for max-sums, the maximization problem is
again completely trivial. Indeed, $\Ga^+(i)\ni i$ implies that all
summands in \eqref{grsum-gen} are bounded above by $1$ and they are equal to $1$ when all $x_i=1$.
We conclude that $\sup\Smax{\Ga}{\xx}=\max\Smax{\Ga}{\xx}=n$. 
\fi
%%%%%%%%%%%%%%%%%%%

\medskip
``There is nothing new under the sun''. 
% [Eccl 1:9]
 Already Daykin \cite{Daykin_1971} considered functions that in our notation are characterized as $S^\Ga_{\sump{1}}(\xx)$ for an out-regular%
\footnote{as defined in \cite{Choudum_1972}}
graph $\Ga$ (that is, $|\Ga^+(v)|=d$, the same for all $v\in\VV(\Ga)$).

The cyclic power sums of Shapiro-Diananda type
\eqref{SDp_gsum} make up a special case of $\fname^\Ga_{\Mp{p}}(\xx)$ where $\Ga$ has an authomorphism that cyclically permutes the nodes. Generally, if the graph $\Ga$ is out-regular, then $\fname^\Ga_{\Mp{p}}$
and $\fname^\Ga_{\sump{p}}$ differ just by a constant factor, which tends to $1$ as $|p|\to\infty$. 

\smallskip
In Definition~\ref{def:grsums} we have put the function $\fname$ in the denominator following the tradition of Shapiro's problem. 
One may want to flip the summed fractions upside down, but the so obtained ``graphic sums of the second kind'' will not widen the scope of the theory. 

Consider the {\it reciprocal function}\ $\widetilde{\fname}$ defined (for $\xx>0$) by
$$
\widetilde{\fname}(x_1,\dots,x_k)=\frac{1}{\fname(x_1^{-1},\dots,x_k^{-1})}.
$$ 
For $\xx>0$, writing $\xx^{-1}=(x_1^{-1},\dots,x_n^{-1})$, we have
$$
\sum_{v\in\VV'(\Ga)}\frac{\widetilde{\fname}(\xx|\Ga^+(v))}{x_v}=
S^{\Ga}_{\fname}(\xx^{-1}).
$$
If $\fname$ is homogeneous of order one, then so is  $\widetilde{\fname}$. If $\fname$ is ascending, then $\widetilde{\fname}$ is descending and conversely. The transformation $\fname\to\widetilde{\fname}$ does not change the type of monotonicity with respect to $\xx$.

Thus we see that $m^{\Ga}_{\fname}$ can be expressed in terms of the ``sum of the second kind'' with reciprocal function, viz.,
\begin{equation}
\label{infSdual}
m^{\Ga}_{\fname}=
\inf_{\xx>0}\sum_{v\in\VV'(\Ga)}\frac{\widetilde{\fname}(\xx|\Ga^+(v))}{x_v}.
\end{equation}

\iffalse
The only possible inconvenience of using formula
\eqref{infSdual} or the restricted domain in \eqref{infSgen}
is that the infimum may be attained on an admissible vector with some $x_i=0$ but not attainable on $\Rplus^n$.
For example, take $S(x_1,x_2)=x_1/\max(x_1,x_2)$. Then any vector
$(x_1,x_2)$ with $\max(x_1,x_2)>0$ is admissible. So $\inf S=\min S=0$ is attained on vectors $(0,x_2)$, $x_2>0$, but they do not lie in $\Rplus^2$. 

We will use both domains, the set of admissible vectors and $\Rplus^n$, as convenient.  
\fi
%%%%%%%%%%%%%%%%%%%%%%%%%%%%%%%%%

%\smallskip
Due to the reciprocity relation
$
\widetilde{\land}=\lor
$
(more generally, $\widetilde{\sump{p}}=\sump{-p}$ and $\widetilde{\Mp{p}}=\Mp{-p}$),   
the greatest lower bounds of graphic max- and min-sums can be represented in two ways,
$$
\begin{array}{l}
\displaystyle
\mmax{\Ga}=\inf_{\xx>0}\sum_{v\in\VV'(\Ga)} \frac{x_v}{\maxx{\xx}{\Ga^+(v)}}
=\inf_{\xx>0}\sum_{v\in\VV'(\Ga)} \frac{\minx{\xx}{\Ga^+(v)}}{x_v},
\\[3ex]
\displaystyle
\mmin{\Ga}=\inf_{\xx>0}\sum_{v\in\VV'(\Ga)} \frac{x_v}{\minx{\xx}{\Ga^+(v)}}
=\inf_{\xx>0}\sum_{v\in\VV'(\Ga)} \frac{\maxx{\xx}{\Ga^+(v)}}{x_v}.
\end{array}
$$
%In the first infima in each line we can either allow $\xx>0$
%(the restricted domain) or any admissible vector $\xx\geq 0$.

\section{Strong reduction}
\label{sec:strong}

In this section we show that the problem of determining the lower bound \eqref{infSgen} for a general graph $\Ga$ reduces to the case of a strongly connected graph assuming $\fname$ satisfies certain natural conditions.
The corresponding {\em formulas of strong reduction}\ are stated in Propositions~\ref{prop:strongred-max} and~\ref{prop:strongred-min} below.

\smallskip
Let $\Ga$ be a simple digraph and $\VV=\VV(\Ga)$. The binary relation
$\succeq\in\VV\times\VV$ (written customarily as $v\succeq v'$ instead of $\succeq(v,v')$) is defined by:
$$
 v\succeq v' 
\;\;\mbox{\rm if $v=v'$ or there exists a (directed) path from $v'$ to $v$}.
$$
The symmetrization of this relation is
$$
 v\eqnod v' \quad\Leftrightarrow \quad v\succeq v'
\;\;\mbox{\rm and}\;\; v'\succeq v.
$$

Let $\VV_1,\dots,\VV_k$ be the equivalence classes of $\VV$ by $\eqnod$.
 The induced subgraphs
$\Ga_i$ with $\VV(\Ga_i)=\VV_i$ are called 
the {\em strong components} of $\Ga$.%
\footnote{Another symmerization of the binary relation $\succeq$ whose equivalence classes are {\em weakly connected components}\ is: 
$ v\stackrel{w}{\eqnod} v' \quad\Leftrightarrow \quad v\succeq v'
\;\;\mbox{\rm or}\;\; v'\succeq v$.
}

The graph $\Ga$ is {\em strongly connected}\ if $k=1$.

The quotient set $\Ga/\eqnod$ is partially ordered by the
relation $\succ$:
$$
\Ga'\succ\Ga''\quad\mbox{\rm if $\Ga'\neq\Ga''$ and
($v'\in\VV(\Ga')$, $v''\in\VV(\Ga'')$) $\Rightarrow$
$v'\succeq v''$.}
$$

The graph $\conds{\Ga}$, called the {\em condencation of $\Ga$}, is obtained from $\Ga$ by collapsing every strong component into a single node. It may contain loops.
The graph $\conds{\Ga}_*$ obtained from $\conds{\Ga}$
by removing loops contains the Hasse diagram of the partial order $\succ$ as a subgraph.

Strong components that are maximal with respect to $\succ$
are called the {\em final strong components}. The set of such strong components (a subset of $\VV(\conds{\Ga})$) will be denoted $\Max\,\conds{\Ga}$
%$\omega(\conds{\Ga})$
 and their union 
(a subgraph of $\Ga$) will be denoted $\lim\Ga$.

By definition, there are no arcs between different components of $\lim\Ga$ and the graph $\conds{\Ga}_*$ is acyclic.

\smallskip
%Let us introduce the notion of height (of a component, a node or a graph) and the operation of partial rescaling of a vector, which will be the main technical device in the proofs concerning the strong reduction.

Let us introduce the notion of height (of a component, a node or a graph), which will be instrumental in the proofs concerning the strong reduction.

The {\em height of a component $\Ga_i\in\VV(\conds{\Ga})$}, to be denoted $h(\Ga_i)$, 
is the maximum length of a chain connecting $\Ga_i$ to $\Max\,\conds{\Ga}_*$. In particular, if $\Ga_i\in\Max\,\conds{\Ga}$,
then $h(\Ga_i)=0$.

The {\em height of a node $v\in\VV(\Ga)$}, denoted $h(v)$, is the height of the component containing it.

The {\em height of the graph $\Ga$}\ is $H(\Ga)=\max_{v\in\VV(\Ga)} h(v)$.

%\smallskip
%Let $\xx\in\Rplus^{\VV}$ be a vector with components indexed by the nodes of the graph $\Ga$.
%For a given $t>0$ and $\Omega\subset\VV$ we define the {\em $(t,\Omega)$-rescaling, $\xx^{(t|\Omega)}$,
%of $\xx$}: it is the vector with components
%$$
% x^{(t|\Omega)}_v=\begin{cases} x_v\;\;\mbox{\rm if $v\notin\Omega$};\\
%t^{h(v)} 
%t x_v\;\;\mbox{\rm if $v\in\Omega$}.
%\end{cases}
%$$

\begin{definition}
Let $\fname$ be a function with variable number of arguments, homogeneous of order $1$.
We say that $\fname$ is a {\em function of max-type}\ if
for $0\leq m\leq k-1$ and any $\xx>0$
$$\sup_{t>0}\fname^{(k)}(x_1,\dots,x_m,tx_{m+1},\dots,tx_k)=+\infty.
$$ 
Equivalently,
$$\sup_{\eps>0}\eps^{-1} \fname^{(k)}(\eps x_1,\dots,\eps x_m,x_{m+1},\dots,x_k)=+\infty.
$$

%$\fname(t\xx)\to+\infty$ as $t\to+\infty$ for any $\xx\in\Rplus^k$.

We say that  $\fname$ is a {\em function of min-type}\ if $\fname$ is descending and
for $0\leq m<k$ and any $\xx>0$
$$\lim_{t\to+\infty}
\fname^{(k)}(x_1,\dots,x_m,tx_{m+1},\dots,tx_k)=  \fname^{(m)}(x_1,\dots,x_m).$$
Equivalently,
$$\lim_{\eps\to 0^+}\eps^{-1} \fname^{(k)}(\eps x_1,\dots,\eps x_m,x_{m+1},\dots,x_k)= \fname^{(m)}(x_1,\dots,x_m).
$$ 
\end{definition}

Recall that $\fname^{(k)}$ denotes the $k$-ary representative of the family of functions with common symbol $\fname$. Since we assume that $\fname$ is symmetric, the $m$ out of $k$ distinguished arguments of $\fname^{(k)}$ in the definition may occupy any positions.

The terminology is defined so as to naturally describe
$\fname=\lor$ as a function of max-type and $\fname=\land$
as a function of min-type. More generally,
the functions $\sump{p}$ and $\Mp{p}$ are of max-type
when $p>0$, as well as $\fname=\Mp{0}$ (the geometric mean). The functions $\sump{p}$, $p<0$, are of min-type.
%Therefore, the value of $m_\fname^\Ga$ in the former case
%is the sum of the values for the final strong components of $\Ga$,
%while in the latter case it is the sum of the values for all strong components of $\Ga$.
%

\begin{remark} If there exist positive constants $c_1,c_2,\dots$ such that $\fname^{(k)}=c_k \mathfrak{g}^{(k)}$,
and one of the functions $\fname$ or $\mathfrak{g}$ 
is of max-type, then so is the other (e.g.\ $\sump{p}$ and $\Mp{p}$, $p>0$).

%It is easy to see that if $\fname$ is an ascending function, homogeneous of order 1 and such that $\fname^{(1)}(1)>0$, then $\fname$ is a function of max-type. We do not have a similarly simple sufficient condition for functions of min-type.

The functions $\Mp{p}$ with $-\infty<p<0$, not being descending,
do not fall in the scope of the above definitions and are thus left out of our analysis of strong reduction.
\end{remark}

\iffalse
\noindent
{\bf Example.}
$$
 S^{\Ga}_{\fname}(x_1,x_2,x_3,x_4)=\frac{x_1}{\Mp{-1}(x_2,x_3)}+\frac{x_2}{\Mp{-1}(x_1,x_3,x_4)}+\frac{x_3}{\Mp{-1}(x_4)}
$$
Explicitly:
$$
 S^{\Ga}_{\fname}(x_1,x_2,x_3,x_4)=\frac{x_1(x_2+x_3)}{2x_2 x_3}+\frac{x_2(x_1x_3+x_1x_4+x_3x_4)}{3x_1 x_3 x_4}+\frac{x_3}{x_4}
$$
and for $\Ga_1$ with $\VV(\Ga_1)=\{v_1,v_2\}$
$$
S^{\Ga_1}_{\fname}(x_1,x_2)=\frac{x_1}{\Mp{-1}(x_2)}+\frac{x_2}{\Mp{-1}(x_1)}=\frac{x_1}{x_2}+\frac{x_2}{x_1}.
$$
Here
$$
 m^{\Ga}_{\fname}\leq
\frac{2}{\sqrt{6}},
\qquad
\sum_{\Ga_i\in\VV(\conds{\Ga})} m^{\Ga_i}_{\fname}=
2+0+0.
$$
\fi

\begin{proposition}
\label{prop:strongred-max}
Let $\fname$ be a function of max-type. Then
$$
 m^{\Ga}_{\fname}=m^{\lim\Ga}_{\fname}=
\sum_{\Ga_i\in\Max\,\conds{\Ga}} m^{\Ga_i}_{\fname}.
$$
\end{proposition}

\begin{proof}
We will prove the proposition by induction on the height $H$ of $\Ga$.

%The base case $H=0$is trivial.
1. In the base case $H=0$ we have $\Ga=\lim\Ga$ by definition. If $\Ga_1,\dots,\Ga_k$ are the components of $\Ga$, then for $v\in\VV(\Ga_i)$ we have
$\Ga^+(v)=\Ga_i^+(v)$. Also, the space $\RR^{\VV(\Ga)}$
is the direct sum of the spaces $\RR^{\VV(\Ga_i)}$, so that any vector $\xx\in\RR^{\VV(\Ga)}$ can be decomposed as
 $\xx=\sum_{i=1}^k \xx^{(i)}$,
$\xx^{(i)}=(x_v)_{v\in\VV(\Ga_i)}$. Therefore  
$$
S^{\Ga}_\fname(\xx)
=\sum_{i=1}^k \sum_{v\in\VV'(\Ga_i)}
\frac{x_v}{\fname(\xx^{(i)}|\Ga_i^+(v))}=\sum_{i=1}^k
S^{\Ga_i}_\fname(\xx^{(i)}).
$$
The minimization over $\xx$ amounts to the minimization over each of the $\xx^{(i)}$ independently.
Hence $m^{\Ga}_{\fname}=\sum_{i} m^{\Ga_i}_{\fname}$.

2. For the induction step, suppose that $H\geq 1$ and $\hat\Ga$ is the graph of height $H-1$
obtained from $\Ga$ by deletion of all components of
height $H$. Clearly, $\lim\hat\Ga=\lim\Ga$.  Hence, by the inductive assumption, 
$m^{\lim\Ga}_{\fname}=m^{\hat\Ga}_{\fname}$.
%the proposition is true for $\hat\Ga$.

The inequality $m^{\Ga}_{\fname}\geq m^{\hat\Ga}_{\fname}$ is obvious, since
$\VV(\hat\Ga)\subset\VV(\Ga)$ and $\hat\Ga^+(v)=\Ga^+(v)$ for any $v\in\VV(\hat\Ga)$.
We have to prove that
$m^{\Ga}_{\fname}\leq m^{\hat\Ga}_{\fname}$.

Let us fix a vector $\xx\in\RR_+^{\VV(\Ga)}$, introduce the parameter $\eps>0$ and define the vector
$\xx'(\eps)$ whose components are
\beq{x-partial-rescale}
 x'_v(\eps)=\begin{cases}
   x_v\;\;\mbox{\rm if $v\in\VV(\hat\Ga)$},\\ %$h(v)<H$
   \eps x_v\;\;\mbox{\rm if $v\in\VV(\Ga)\setminus\VV(\hat\Ga)$}. %$h(v)=H$
\end{cases}
\eeq
We have $\Ga^+(v)={\hat\Ga}^+(v)$ for any $v\in\VV(\hat G)$, so
\beq{redsum-inv}
 \sum_{v\in\VV'(\hat\Ga)}\frac{x'_v(\eps)}{\fname(\xx'(\eps)\mid\Ga^+(v))}
=\sum_{v\in\VV'(\hat\Ga)}\frac{x_v}{\fname(\xx\mid{\hat\Ga}^+(v))}.
\eeq
From the definition of functions of max-type it follows that 
$$
 \inf_{\eps>0}\sum_{v\in\VV'(\Ga)\setminus\VV(\hat\Ga)}\frac{x'_v(\eps)}{\fname(\xx'(\eps)\mid\Ga^+(v))}=0.
$$
We conclude that
$$
 \inf_{\eps>0} S^{\Ga}_{\fname}(\xx'(\eps))=
S^{\hat\Ga}_{\fname}(\hat\xx),
$$
where $\hat\xx$ is the truncation of the vector $\xx$ containing only the components with indices from $\VV(\hat\Ga)$.
Therefore $m^{\Ga}_{\fname}\leq m^{\hat\Ga}_{\fname}$.
\end{proof}

\begin{proposition}
\label{prop:strongred-min}
Let $\fname$ be a function of min-type. 

(a) If $\hat\Ga$ is a subraph of $\Ga$ 
{\em(that is, $\VV(\hat\Ga)\subset\VV(\Ga)$ and ${\hat\Ga}^+(v)\subset\Ga^+(v)$ for any $v\in\VV(\hat\Ga)$)}, then
$$
m^{\Ga}_{\fname}\geq m^{\hat\Ga}_{\fname}.
$$

(b) For any graph $\Ga$
$$
 m^{\Ga}_{\fname}=
\sum_{\Ga_i\in\VV(\conds{\Ga})} m^{\Ga_i}_{\fname}.
$$
\end{proposition}

\begin{proof}
(a) By definition, a function of min-type is descending. 
Therefore $\fname(\xx|\Ga^+(v))\leq \fname(\xx|{\hat\Ga}^+(v))$ for any
$\xx\in\RR_+^{\VV(\Ga)}$. Hence $S^{\Ga}_\fname(\xx)\geq S^{\hat\Ga}_\fname(\hat\xx)$, where $\hat\xx$ is the truncation of the vector $\xx$
containing only the components with indices from $\VV(\hat\Ga)$. The claimed inequality between the lower bounds follows.

\smallskip
(b) We will carry out the proof by induction on the number $N=|\VV(\conds{\Ga})|$ of strong components.
The base case $N=1$ is trivial.

Suppose $N>1$ and let $\Ga_1$ be one of the components of $\Ga$ of the maximum height $H=H(\Ga)$. Let $\hat\Ga$ be the induced subgraph of $\Ga$ on the set of nodes
$\VV(\hat\Ga)=\VV(\Ga)\setminus\VV(\Ga_1)$. In view of the induction hypothesis, we need to prove that
$m_\fname^\Ga=m_\fname^{\Ga_1}+m_\fname^{\hat\Ga}$.
By (a), the left-hand side is not less than the right-hand side. The remaining task is to prove the inequality $m_\fname^\Ga\leq m_\fname^{\Ga_1}+m_\fname^{\hat\Ga}$.

The argument is similar to that used in the proof of Proposition~\ref{prop:strongred-max}.
Again, we fix a vector $\xx$ and define $\xx'(\eps)$ by the formula \eqref{x-partial-rescale}. 

If $v\in\VV(\Ga_1)$, 
%and $\hat\xx$ is the same as in the final part of the previous proof, 
then by the definition of a function of min-type we have
$$
\lim_{\eps\to 0^+}
\eps^{-1}\fname(\xx'(\eps)\mid\Ga^+(v))\geq \fname(\xx\mid\Ga_1^+(v)).
$$
Hence
$$
  \lim_{\eps\to 0^+}\sum_{v\in\VV'(\Ga_1)}\frac{x'_v(\eps)}{\fname(\xx'(\eps)\mid\Ga^+(v))}
\leq\sum_{v\in\VV'(\Ga_1)}\frac{x_v}{\fname(\xx\mid\Ga_1^+(v))}=S_\fname^{\Ga_1}(\xx^{(1)}),
$$
where $\xx^{(1)}$ is the truncation of the the vector $\xx$ containing only the components with indices from $\VV(\Ga_1)$. 

Taking into account the formula \eqref{redsum-inv}, which remains valid here, we get
$$
 \lim_{\eps\to 0^+}S_\fname^{\Ga}(\xx'(\eps))\leq S_\fname^{\Ga_1}(\xx^{(1)})+S_\fname^{\hat\Ga}(\hat\xx).
$$
Minimizing 
%independently 
over $\hat\xx$ and $\xx^{(1)}$, we complete the induction step.
\end{proof}

%Poset of graphs with fixed $\VV$; behaviour of f.
%[from: Simple prelims]

\section{Admissible maps and functional graphs}
\label{sec:unary}

A {\em functional graph}\ is an out-regular graph
with $d^+(v)\equiv 1$. In other words, in such a graph
the set $\Ga^+(v)$ is a singleton for every node $v$.
The term {\em functional graph}\ refers to the fact that $\Ga^+(v)=\{\sigma(v)\}$ with some function $\sigma:\VV(\Ga)\to\VV(\Ga)$.

A functional graph is strongly connected if and only if $\sigma$ is a cyclic permutation. The weaker condition:
$\sigma$ is a bijection (a permutation) --- is equivalent to the equality $\Ga=\lim\Ga$.

\smallskip
For the analysis of lower bounds for max-sums and min-sums
the notion of an admissible map will be instrumental.

\begin{definition}
\label{def:adm_map}
Given a graph $\Ga$,
a map $\sigma:\,\VV(\Ga)\to\VV(\Ga)$
is called {\em admissible}\ (or {\em $\Ga$-admissible}\ if one needs to be precise about the underlying graph), if $\sigma(v)\in\Ga^+(v)$ for any $v\in\VV(\Ga)$.
\end{definition}

For an arbitrary map $\sigma:\,\VV(\Ga)\to\VV(\Ga)$ we can define the functional graph $\Ga_\sigma$ on the set of nodes $\VV(\Ga_\sigma)=\VV(\Ga)$ by putting
$\Ga_\sigma^+(v)=\{\sigma(v)\}$. The map $\sigma$ is $\Ga$-admissible if and only if $\Ga_\sigma$ is a subgraph of $\Ga$.

The set of all $\Ga$-admissible maps will be denoted
$\mathcal{F}_\Ga$. 

Clearly, $\mathcal{F}_\Ga=\emptyset$ if and only if
$\min_{v\in\VV(\Ga)} d^+(v)=0$ (never the case for strongly connected graphs with more than one node)
and $|\mathcal{F}_\Ga|=1$ if and only if $\Ga$ is a functional graph.

Any function $\sigma\in\mathcal{F}_\Ga$ is a {\em choice function}\ for the family of sets $\{\Ga^+(v),v\in \VV(\Ga)\}$ in the sense of set theory.

\begin{remark} 
The set $\mathcal{F}_\Ga$ can also be characterized as the set of bases of the \emph{tail partition matroid}\ of the graph $\Ga$ (by identifying a function $\kappa\in\mathcal{F}_\Ga$ with the set of arcs
$\{(a,\kappa(a)),\, a\in\VV(\Ga)\}\subset\arcs{\Ga}$) 
\cite[Example~8.2.22, p.~357]{West_2001}.
\end{remark}

Let us call surjective (equivalently, bijective) admissible functions {\em admissible bijections}. The set of such functions will be denoted $\mathcal{F}^*_\Ga$.
An admissible bijection provides a {\em system of distinct representatives}\ (SDR)
%\footnote{also called a {\em transversal}}, 
for the family of sets $\{\Ga^+(v)\}$. Thus $\mathcal{F}^*_\Ga\neq\emptyset$ if and only if the system $\{\Ga^+(v)\}$ admits an SDR.

By Hall's theorem, $\mathcal{F}^*_\Ga\neq\emptyset$ if and only if for any subset $\mathcal{U}\subset\VV(\Ga)$
there holds $\left|\cup_{v\in\mathcal{U}}\Ga^+(v)\right|
\geq|\mathcal{U}|$.

\smallskip
The relevance of admissible functions is explained by the following simple lemma.

\begin{lemma}
\label{lem:unary-red}
Let $\Ga$ be a graph with $\min_v d^+(v)>0$.
Let $\fname=\lor$ or $\land$. Then for any vector
$\xx\in\RR^{\VV(\Ga)}$ there exists a function
$\sigma\in\mathcal{F}_\Ga$ such that
\beq{sum-unary}
 S_\fname^\Ga(\xx)=S_\fname^{\Ga_\sigma}(\xx)=
\sum_{v\in\VV(\Ga)} \frac{x_v}{x_{\sigma(v)}}.
\eeq
\end{lemma}

\begin{proof}
For $\fname=\lor$ or $\land$ there is always some
$w\in\Ga^+(v)$ such that $x_v=\fname(\xx|\Ga^+(v))$.
Put $\sigma(v)=w$. (If there are several possibilities, chose one arbitrarily.)
\end{proof}

The function $\sigma$ in \eqref{sum-unary} will be called
$(\xx,\Ga)$-admissible or simply $\xx$-admissible if there is no ambiguity.
Of course it may (and usually does) vary with $\xx$, so in the problem of minimization over $\xx$ Lemma~\ref{lem:unary-red} does not provide a reduction from a given graph $\Ga$ to a specific functional graph. Yet even the ``$\xx$-dependent reduction'' will be useful in both max and min cases.
In Sec.~\ref{sec:min} we will need to describe the $\xx$-dependence in \eqref{sum-unary} more precisely.

\smallskip
In the remaining part of this section we find
$m_\fname^\Ga$ for $\fname=\Mp{p}$,
$p\in[-\infty,+\infty]$, and a functional graph $\Ga$. 

Let $\sigma$ be the (unique) $\Ga$-admissible function.
We have $\Mp{p}(\xx|\Ga^+(v))=x_{\sigma(v)}$ for any $p$. So min-sums and max-sums are the same, as well as any $p$-sums, and have the form as in the right-hand side of
\eqref{sum-unary}.

To determine the greatest lower bound of this sum, it suffices, in view of the results of Sec.~\ref{sec:strong}, to consider the case of a strongly connected graph $\Ga$.
In this case, as we noted earlier, $\sigma$ is a cyclic permutation of the set $\VV(\Ga)$. By the AM-GM inequality, the minimum value of the sum is $|\VV(\Ga)|$, attained when all $x_v$ are equal. 

Using Proposition~\ref{prop:strongred-max}, we can drop the assumption of strong connectedness and come to the following result.

\begin{proposition}
\label{prop:unary}
If $\Ga$ is a functional graph and $\fname=\Mp{p}$ with 
any $p\in[-\infty,+\infty]$, then
$$
 m^\Ga_{\fname}=|\VV(\lim\Ga)|.
$$
In particular, this is true for $\fname=\lor$ and $\fname=\land$.
\end{proposition}

\section{Max-sums}
\label{sec:max}

It turns out that in the problem of minimization of max-sums the answer is always an integer.

\begin{theorem}
\label{thm:max-girth}
If $\Ga$ is a strongly connected graph, then 
$m^\Ga_\lor$ is equal to $g(\Ga)$, the {\em girth}\ of $\Ga$,
which is the minimum length of a (directed) cycle in $\Ga$.

If $\Ga$ is an arbitrary graph, then $m^\Ga_\lor$ is equal to $g^*(\Ga)$, 
%{\em the final girth of $\Ga$}\, that is, 
the sum of girths of all components of $\lim\Ga$. 
\end{theorem}

\begin{proof}
The general result follows from the result for strongly connected graphs by Proposition~\ref{prop:strongred-max}.

Below, till the end of the proof we assume that $\Ga$ is strongly connected.

%-----------------------------
\iffalse
P.O. on the nodes of digraph: $b>>a$ if there is path from $a$ to $b$ but not from $b$ to $a$. 

Quotient of the set of closed elements of this relation is the set of closed elements of the condensation of $\Ga$.

Poset of all $\Ga$-permitted maps: $\kappa\succ\kappa'$ iff $\forall a$ $\kappa(a)\succeq\kappa'(a)$.
\fi
%-----------------------------

Let $\VV=\VV(\Ga)$. 
Suppose $\xx\in\RR_+^{\VV}$. Let $\sigma:\,\VV\to\VV$ be an $\xx$-admissible function, so that 
\eqref{sum-unary} holds.
By Proposition~\ref{prop:unary}, $S_\lor^{\Ga_\sigma}(\xx)
\geq|\VV(\lim\Ga_\sigma)|$.
Since the map $\sigma$ is bijective on the set $\VV(\lim\Ga_\sigma)$, this set is a union of cycles of 
$\sigma$. Hence $|\VV(\lim\Ga_\sigma)|\geq g(\Ga)$.
We conclude that $m_\lor^\Ga\geq g(\Ga)$. 

Let us now take a cycle in $\VV(\Ga)$ of length $g=g(\Ga)$.
We may assume that it consists of the nodes
$v_1,\dots,v_g$ and the arcs $v_1\to v_2,\dots, v_g\to v_1$.

We will define a special $\Ga$-admissible map $\sigma$.
Denote $V_0=\{v_1,\dots,v_g\}$. Put $\sigma(v_1)=v_2,
\dots, \sigma(v_g)=v_1$. Thus $\sigma$ is defined on $V_0$. 

If $V_0\neq\VV$, then denote by $V_1\subset\VV\setminus V_0$ the set of nodes
$v$ such that there is an arc from $v$ to $V_0$. 
If $v\to v_i$ is such an arc (arbitrarily chosen, if there is more than one suitable $v_i$), put $\sigma(v)=v_i$.
Now $\sigma$ is defined on $V_0\cup V_1$.

If $V_0\cup V_1\neq \VV$, then we take $V_2$, the set of nodes from which there is an arc to $V_1$, and define $\sigma$ on $V_2$ similarly to the above.
We continue this process until the set $\VV$ is exhausted.
It must happen, since from any node of $\Ga$ there exists a path to $V_0$. 

As the result, we obtain the map $\sigma$ such that
$\sigma:\,V_{j+1}\to V_j$ for $j>0$. Moreover, 
if $v\in V_j$, then the shortest path from $v$ to $V_0$
consists of $j$ arcs.

Take $\eps\in(0,1)$ and define the vector $\xx=\xx(\eps)\in\RR_+^\VV$ as follows:
$$
 x_v=\eps^j\quad\text{whenever $v\in V_j$,\;$j\geq 0$}.
$$
It is easy to see that for $v\in V_j$, $j>0$, the set $\Ga^+(v)$ contains a node from $V_{j-1}$ but does not contain nodes from $V_{j-k}$ with $k>1$. Hence
$\lor(\xx\mid\Ga^+(v))=\eps^{j-1}$.
Obviously, $\lor(\xx\mid\Ga^(v))=1$ for $v\in V_0$.

Thus we have
$$
 S_\lor^\Ga(\xx)=\sum_{v\in V_0} 1+\sum_{v\in\VV\setminus V_0} \eps.
$$
Letting $\eps\to 0^+$ we see that $m_\lor^\Ga\leq |V_0|=g$.
This completes the proof.
\end{proof}

\begin{remark}
Let us provide initial directions for an algorithmically minded reader concerning the complexity of computing girth of a strongly connected graph with $n$ nodes.
The most obvious method is to compute successive powers of the adjacency matrix until its trace
becomes positive \cite[Exercise~3.22]{BangGutin_2002}. It takes
%A Shortest cycle in digraph: 
$O(n^4)$ operations in the worst case.
%(In \S~5.2 of the same reference computing strong components is discussed.)
Better algortihms exist, see e.g.\ \cite{ItaiRodeh_1978} for an $O(n^{\log_2 7}\ln n)$ algorithm.
\end{remark}

We take a short recess to pay some attention to graphic power sums
with exponents $p<\infty$.
The evaluation of $m^{\Ga}_{\sump{p}}$ or $m^{\Ga}_{\Mp{p}}$ for an arbitrary $p<\infty$ can be difficult even for graphs with small number of nodes. 
In the absence of anything better, we state a simple, rough double-sided estimate.

\begin{proposition}
\label{prop:grpsum-ineq}
For any graph $\Ga$, we have the estimates
\begin{equation}
\label{gr-psum-lbnd}
g^*(\Ga)\leq m_{\Mp{p}}^\Ga\leq |\VV(\Ga)|, \quad p\in[-\infty,+\infty]
\end{equation}
and
\begin{equation}
\label{gr-psum-ubnd}
 m_{\Mp{p}}^\Ga\leq |\VV(\lim \Ga)|, \quad p\in[0,+\infty].
\end{equation}
\end{proposition}

\begin{proof}
The left estimate in \eqref{gr-psum-lbnd} follows from Theorem~\ref{thm:max-girth} and monotonicity of the function $p\mapsto \Mp{p}(\xx|\Ga^+(v))$ for any $\xx\in\RR_+^{\VV(\Ga)}$ and any $v\in\VV(\Ga)$.

The right estimate in \eqref{gr-psum-lbnd} is obvious: 
putting $\xee=(1,1,\dots,1)$, we get
$$
m_{\Mp{p}}^\Ga\leq S^\Ga_{\Mp{p}}(\xee)=
|\VV(\Ga)|.
$$ 

Using Proposition~\ref{prop:strongred-max} and 
applying this trivial estimate to 
the graph $\lim\Ga$ in place of $\Ga$, we obtain \eqref{gr-psum-ubnd}.
\end{proof}

Returning to max-sums ($p=+\infty$), note some special cases of Theorem~\ref{thm:max-girth}. 

\smallskip
1. If $\Ga$ is a strongly connected graph containing at least one loop (a node $v$ such that $v\in\Ga^+(v)$), then  $m^\Ga_\lor=1$.

\smallskip
2. Let us examine cyclic max-sums $S_{n;J,+\infty}(\xx)$ in light of Theorem~\ref{thm:max-girth}.
For every natural $n$ the relevant graph $\Ga$ is circulant: its set of nodes is $\VV(\Ga)=[n]$ and the node-arc adjacencies are defined by
$\Ga^+(i)=(J+i)\pmod n$, $i=1,\dots,n$.

\smallskip
a) Suppose that $1\in J$. In view of the special case 1
and of Proposition~\ref{prop:strongred-max}, 
$\inf_{\xx}S_{n;J,+\infty}(\xx)$ is equal to the number of
components, $q$, of the graph $\Ga$. Since the cyclic shift $\tau$ (see the beginning of Sec.~\ref{sec:cycsum}) is an automorphism of the graph $\Ga$, the components
are $\tau$-congruent and we have $q=n/m$, where $m$ is the size of the component. It can be characterized as follows. Let $\ZZ_n$ be the abelian group $\ZZ/n\ZZ$.
Let $G_J$ be its subgroup generated by the elements
$\{j-1\mid j\in J\}$. Then $m=|G_J|$. Clearly, if $|J|>1$,
then $m=n$ (so $q=1$) for infinitely many $n$ (for example, for $n$ coprime with $j-1$ for some $j\in J\setminus\{1\}$).

\smallskip
b) Suppose now that $1\notin J$. Let $G_J$ be defined the same way as in (a) and let $t$ denote the minimum number of summands (not necessarily distinct) in a sum $(j_1-1)+\dots+(j_t-1)$ representing zero in $\ZZ_n$ with $j_i\in J$.
The number of $G_J$-cosets, $n/|G_J|$, is equal to the number of strong components of $\Ga$
and $t$ is the girth of each of the components. Hence,
by Theorem~\ref{thm:max-girth},
$$m^\Ga_\lor=t\cdot \frac{n}{|G_J|}.$$

If $J\subset[2,b]$ and $n>b$, then
obviously $t\geq\lfloor (n-1)/(b-1)\rfloor+1=\lfloor (n+b-2)/(b-1)\rfloor$.
%
%If $2\in J$, then $G_J=\ZZ_n$. 
We obtain the estimate
\beq{DiaDay2}
 \inf_\xx S_{n;J,+\infty}(\xx)=m_\lor^\Ga\geq \left\lfloor\frac{n+b-2}{b-1}\right\rfloor.
\eeq
If, moreover, $J=[2,k+1]$ is an interval, then
$G_J=\ZZ_n$, so $\Ga$ is strongly connected.
In this case the right-hand side of \eqref{DiaDay2} is equal to the number $t$ of the summands
in the representation
$n=k+\dots+k+k'$, with $k'=n-(t-1)k$. Hence $t=g(\Ga)$
and the inequality in  \eqref{DiaDay2} turns to equality.
Thus we obtain the limit case ($\nu=+0$) of Diananda's 
formula \eqref{DiaDay}.

%\begin{example}
%Circulant graph \cite[p.~306]{West_1996}
%(West1996) $C_{n,k}$; 
% (Here for odd $k=2r+1$, $J=[1-r,1+r]$.)
%
%$g(C_{n,k})=\lceil n/k\rceil$.
%``Circulant shell'': same girth, though outdegree is $2$ for each node.
%\end{example}

\smallskip
Let us summarize the discussion of the special cases 2(a) and 2(b) in a slightly cruder but more transparent form, emphasizing in (b) the behaviour for large $n$.

\begin{proposition}
\label{prop:maxsum-cyclic}
Let $m_{\lor}(n,J)=\inf_{\xx>0} S_{n,J,+\infty}(\xx)$
be the greatest lower bound of the cyclic max-sum with $n$ terms defined by the pattern $J\subset\ZZ$.

\smallskip
(a) Suppose that $J\ni\{1\}$, $|J|>1$, and $\gcd(j-1\mid 1\neq j\in J)=s$. Then 
$$
m_{\lor}(n,J)=\gcd(n,s).
$$
In particular, $m_{\lor}(n,J)=1$ infinitely often as $n\to\infty$.

\smallskip
(b) If $J\not\ni\{1\}$ and $r(J)=\max(|j-1|,\,j\in J)$, then
$$
m_{\lor}(n,J)=\frac{n}{r(J)}+O(1), \quad n\to\infty.
$$
\end{proposition}

As a simple consequence of part (b) and the
trivial inequality $\Mp{p}(\xx|\Omega)\geq |\Omega|^{-1/p}\cdot\max(\xx|\Omega)$ for $p>0$, we obtain the following asymptotic upper bound for minimum values of cyclic $p$-sums with an arbitrary pattern $J\not\ni\{1\}$:
$$
\inf_{\xx>0} S_{n,J,p}(\xx)\leq \frac{n|J|^{1/p}}{r(J)}+O(1)\leq 2n|J|^{\frac1p-1}+O(1)
$$
as $n\to\infty$.
In particular, as the there is no upper bound for $r(J)$ stipulated by the size of $J$, we at once reject the proposition that
there might exist a pattern-independent  (``universal'') positive lower bound $A$ in the inequality
$
n^{-1}S_{n,J,1}(\xx)
\geq A
$.

\smallskip
We conclude this section with some further upper estimates for girth $g(\Ga)$ and ``total girth''
$g^*(\Ga)$ that appears in Theorem~\ref{thm:max-girth}.

Caccetta and H\"aggkvist cojectured the upper bound $g(\Ga)\leq\lceil n/k\rceil$ for any strongly connected graph $\Ga$ with $n$ nodes and minimum outdegree $k$. 
See \cite{CacHag_1978} and \cite[Conjecture 8.4.1 (p.~330)]{BangGutin_2002}.
In our context the CH conjecture implies that for such a graph $\Ga$ and for some $\xx>0$ the inequality $\Smax{\Ga}{\xx}\leq(n+k-1)/k$
holds. Moreover, this (obviously) remains true for any graph with one final strong component, i.e.\ for any weakly connected digraph. 
If $\Ga$ is a digraph with weak components $G_1,\dots,G_m$
and $|\VV(G_j)|=n_j$, $j=1,\dots,m$, then
$$
g^*(\Ga)\leq\sum_{j=1}^m\frac{n_j+k-1}{k}=\frac{n+m(k-1)}{k}
\leq\left\lceil\frac{n-(m-1)}{k}\right\rceil+(m-1).
$$
Clearly, $n_j\geq k$ for all $j$, hence $m\leq n/k$ We get the estimate conditional on the CH conjecture:
%\beq{maxsum-valency-lbnd}
$$
\mmax{\Ga}=g^*(\Ga)<\frac{2n}{k}
$$
%\eeq
for any digraph $\Ga$ (weakly connected or not).
In particular, the $n$-dimensional vector $\xx=\xee=(1,1,\dots,1)$ is very far from optimal (one that minimizes $\Smax{\Ga}{\cdot}$) if $k\gg n$.
We see a sharp contrast with situation observed in the case $p=1$, at least for Diananda's sums: the result of \cite{Sadov_2016}
shows that $S_{n,k,1}(\xee)$ differs from $\inf_{\xx>0} S_{n,k,1}(\xx)$
only by a moderate constant.

\iffalse
It is instructive to compare the inequality 
\eqref{maxsum-valency-lbnd} 
to some estimates available in the case $p=1$.
On the one hand, we have the trivial estimate
$S^{\Ga}_{\mathrm{sum}_1}(\xee)\leq n/k$ for an out-regular graph $\Ga$ with $d^+(v)\equiv k$.
%(cf.\ also \eqref{gr-psum-lbnd} and \eqref{Mp_vs_sump}).
On the other hand, the result of \cite{Sadov_2016}
shows that $\inf_{\xx>0} S^{\Ga}_{\mathrm{sum}_1}(\xx)\asymp
n/k$ for special cyclic graphs (corresponding to the Diananda sums).
%in this case ($p=1$) the upper bound
%in general cannot be improved beyond possible amelioration of the constant.
In a very vague sense, we can say that the lower bounds for max-sums behave ``similarly'' to those of 
$1$-sums $S^\Ga_{\mathrm{sum}_1}$. Changing the normalization of the denominators (i.e.\ passing from $S^\Ga_{\mathrm{sum}_1}$ to $S^\Ga_{\mathfrak{M}_1}$)
leads to a drastically different behavior. As fuzzy 
as this argument reads, it makes us to favor the sums 
$S^\Ga_{\mathrm{sum}_p}$ rather than $S^\Ga_{\mathfrak{M}_p}$ as likely ``more natural'' version of graphic power sums.
\fi

Finally, we mention one unconditional result
that bounds girth from above if  
a lower estimate for the number of arcs is known.
%(of order $n^2$ where $n=|\VV(\Ga)|$).

Theorem 8.4.7 in \cite[p.~331]{BangGutin_2002},
due to Bermond, Germa, Heydemann and Sotteau (BGHS), says:  {\em if $\Ga$ is a strongly connected graph, $|\VV(\Ga)|=n$, 
and an integer $t \geq 2$ is such that
$$
|\arcs{\Ga}| \geq \frac{(n-t)(n-t+1)}{2}+n,
$$
then $g(\Ga) \leq t$.} 

\smallskip
The BGHS's unconditional estimate is in most cases much weaker than the one conjectured by Caccetta and H\"aggkvist. 

\begin{example}
Let $\Ga$ be a strongly connected graph with
$n=40$ nodes 
%containing the cycle $1\to 2\to\dots\to n\to 1$ (to guarantee strongly connected) 
and out-regular with $d^+(i)\equiv d=12$. 
The number
of arcs in $G$ is $A=nd=480$, and 
the assumption of the theorem is met with $t=10$, 
since $\binom{30}{2}=435<A-n=440$.
So the BGHS theorem gives the bound $g(\Ga)\leq 10$.

The CH conjecture suggests that
$g(\Ga)\leq \lceil 40/12\rceil=4$.

By Theorem~\ref{thm:max-girth} this means that if $\Omega_1,\dots,\Omega_{40}$
are any $12$-element subsets of $[1,40]$ such that 
(for instance) $\Omega_i\ni(i+1)\mod 40$ (to guarantee strong connectedness), then there exists 
$\xx\in\RR_+^{40}$ such that 
$$
 \sum_{i=1}^{40}\frac{x_i}{\max(x_j\mid j\in\Omega_i)}
\leq 4.
$$
\end{example}

\section{Min-sums I: The problem for an individual graph}
\label{sec:min}

We turn to the minimization problem for min-sums.
In this section we discuss how the problem can be solved {\it in principle} for the given graph $\Ga$.

We begin with an estimate for $m_\land^\Ga$ in combinatorial terms with sign opposite to that in \eqref{gr-psum-ubnd} 
(in the present case $p=-\infty$).
Recall that the set $\mathcal{F}_\Ga$ of admissible maps is defined in Sec.~\ref{sec:unary}, Definition~\ref{def:adm_map}.

\iffalse
\begin{remark}
Question: estimate the upper bound for $\mmin{\Ga}$ if the rank $r_\Ga<n$ of the {\em transversal matroid}\ 
\cite{MatroidText} for the set system $\{\Ga^+(v)\}_{v\in\VV(\Ga)}$ is given. 
Note that $r_\Ga$ is not less than the right-hand side of 
\eqref{lbnd-minsum-sigma}.
\end{remark}
\fi

\begin{proposition} 
\label{prop:lbnd-minsum-sigma}
For any graph $\Ga$ 
\begin{equation}
\label{lbnd-minsum-sigma}
 \mmin{\Ga}\geq\max_{\sigma\in\mathcal{F}_\Ga}|\VV(\lim\Ga_\sigma)|.
\end{equation}
\end{proposition}

\begin{proof}
(a) Since $\Ga_\sigma^+(v)\subset\Ga^+(v)$, we have
$\mmin{\Ga}\geq\mmin{\Ga_\sigma}$, cf.\ Proposition~\ref{prop:strongred-min}(a).
But $\mmin{\Ga_\sigma}=|\VV(\lim\Ga_\sigma)|$ by Proposition~\ref{prop:unary}. The result follows.
\end{proof}

The case of cyclic min-sums now looks completely trivial.

\begin{proposition}
\label{prop:minsum-cyclic}
For $p=-\infty$, any $n\geq 1$, and any finite nonempty set $J\subset\ZZ$ the inequality in \eqref{triv_est_SnJp} turns to equality. That is,
$ \inf_{\xx} S_{n;J,-\infty}(\xx)= n$.
\hfill$\Box$
\end{proposition}

On the contrary, the general problem of finding $\mmin{\Ga}$, to be discussed now, does not appear to be simple at all.

Some preparations are required before we can  formulate Theorem~\ref{thm:chamber_optimization}.
Essentially they are aimed at a new level of understanding of the formula \eqref{sum-unary}. 

\smallskip
A {\it preferential arrangement}, or a {\em ballot}, 
%{\it block-order}\
% упорядоченные перестановки (в переводе Stanley), ранжирования
on a nonempty set $\Omega$ is a partition of $\Omega$ into blocks with a linear order on the set of the blocks.

For example, the set $\{1,2,3\}$ can be block-partitioned
as $(123)$, $(12)(3)$, $(13)(2)$, $(1)(23)$, and $(1)(2)(3)$.
A partition with $k=1,2,3$ blocks can be ordered in $k!$ ways,
which gives us the total of $1\cdot 1+3\cdot 2!+1\cdot 3!=13$
preferential arrangements on the 3-element set.

For a given graph $\Ga$, denote the set of all preferential arrangements of the set $\VV(\Ga)$ by $\Arr_\Ga$.

To every preferential arrangement $\mathfrak{A}\in\Arr_\Ga$ we put in correspondence certain subset $W_{\mathfrak{A}}$ of $\RR_+^{\VV(\Ga)}$.
Namely,
$W_{\mathfrak{A}}$ consists of vectors $\xx$ such that

\smallskip
(i) $x_{v}=x_{v'}$ if $v$ and $v'$ lie in the same block of $\mathfrak{A}$; 

\smallskip
(ii) If $v\in B$, $v'\in B'$, where
$B$ and $B'$ are distinct blocks and $B>B'$, then 
$x_v>x_{v'}$.

\smallskip
For example, if the nodes of $\Ga$ are labelled as $1,2,3$
and $\mathfrak{A}$ is
the preferential arrangement $(1,3)>(2)$, then 
$$
 W_{(13)>(2)}=\{\xx=(x_1,x_2,x_3)\mid x_1=x_3>x_2>0\}.
$$

For any vector $\xx>0$ there is a unique
preferential arrangement $\mathfrak{A}$ such that $\xx\in W_\mathfrak{A}$. We call $\mathfrak{A}$ the $\xx$-type arrangment.

\begin{definition}
\label{def:prefarr_red}
Let $\mathfrak{A}\in\Arr_\Ga$ be a preferential arrangement with blocks $B_1>\dots>B_k$.
Introduce the $k$ variables
$y_1,\dots,y_k$ corresponding to the blocks and put
$\yy=(y_1,\dots,y_k)$. Define the functions
$\alpha$ and $\beta$ from $\VV(\Ga)$ to $[k]$ as follows:

$\alpha(v)=j$ if $v\in B_j$;

$\beta(v)=j$ if $B_j$ is the $\mathfrak{A}$-minimal block among those that have nonempty intersection with $\Ga^+(v)$.

The {\em $\mathfrak{A}$-reduction}\ of the min-sum
$\Smin{\Ga}{\xx}$ is the function
\begin{equation}
\label{pa-qsum}
\Smin{\Ga}{\yy|\mathfrak{A}}=\sum_{v\in\VV(\Ga)}
\frac{y_{\alpha(v)}}{y_{\beta(v)}}.
\end{equation}
\end{definition}

\begin{lemma}
\label{lem:prefarr-red}
Suppose $\xx>0$ and $\mathfrak{A}\in\Arr_\Ga$ is an
$\xx$-type arrangement.
%$\xx\in W_{\mathfrak{A}}$. 
Let $\yy$ be the vector defined by $y_{\alpha(v)}=x_v$
(in the notation of Definition~\ref{def:prefarr_red}). Then
$$
 \Smin{\Ga}{\yy|\mathfrak{A}}=\Smin{\Ga}{\xx}.
$$
\end{lemma}

\begin{proof}
Note first that $y_j$ are defined correctly: if $j=\alpha(v)=\alpha(v')$ then $x_v=x_{v'}$ by definition of the set $W_{\mathfrak{A}}$.

It remains to see that $y_{\beta(v)}=\land(\xx|\Ga^+(v))$;
this follows from the definitions of the function $\beta$ and the set $W_{\mathfrak{A}}$.
\end{proof}

\begin{remark}
The assumptions of Lemma~\ref{lem:prefarr-red} imply the inequalities $y_1>\dots>y_k$. Conversely, if $\mathfrak{A}$ is some preferential arrangement, as in Definition~\ref{def:prefarr_red}, and $\yy\in\RR_+^k$ is a vector satisfying this inequalities, then the vector
$\xx$ defined by $x_v=y_{\alpha(v)}$, $v\in\VV(\Ga)$,
lies in $W_\mathfrak{A}$.
\end{remark}

In view of Lemma~\ref{lem:prefarr-red}, one would expect that the minimization problem for $\Smin{\Ga}{\cdot}$ can be reduced to the analysis of the sums \eqref{pa-qsum}. Or, one may ask, --- maybe --- just to the analysis of the unary sums $\Smin{\Ga_\sigma}{\xx}$, see \eqref{sum-unary}?

The latter expectation is too naive and false. However the former can indeed be carried through to the end. This will make up the 
%most of the 
remaining part of this section.
% (before Theorem~\ref{thm:min-sdr}). 
After necessary preliminaries we prove the main reduction theorem (Theorem~\ref{thm:chamber_optimization}) and then state an algorithm for finding $\mmin{\Ga}$.

We emphasize that the nature of the sum \eqref{pa-qsum} is more general than that of the unary sums: the same $y$-variable may occur in the numerators of several different terms. Let us describe the situation more formally.

Suppose a graph $\Ga$ and a preferential arrangement
$\mathfrak{A}$ on the set $\VV=\VV(\Ga)$ are fixed.
Define a new graph $\Ga_\mathfrak{A}$ as follows.

The set of nodes $\VV(\Ga_\mathfrak{A})$ is indexed by the blocks of $\mathfrak{A}$. That is, each node corresponds to one of the $y$-variables in Definition~\ref{def:prefarr_red}. 

The set of arcs $\AA(\Ga_\mathfrak{A})$ is in one-to-one correspondence with the set of nodes $\VV(\Ga)$: the node $v\in\VV$ gives rise to the arc $\alpha(v)\to\beta(v)$ of %the graph 
$\Ga_\mathfrak{A}$.
(Multiple arcs were not allowed in $\Ga$, but they may be present in $\Ga_\mathfrak{A}$. For instance, if $\mathfrak{A}$ consists of a single block, then $\Ga_\mathfrak{A}$ contains a single node and $|\VV(\Ga)|$ loops.)

Let us label the nodes of the graph $\Ga$ with numerical values: $v\mapsto x_v$. The induced labelling of the nodes of $\Ga_\mathfrak{A}$ is:  $B_j\mapsto y_j=x_{\alpha(j)}$. Moreover, there is an induced labelling of the {\em arcs}\ of $\Ga_\mathfrak{A}$: the arc $\alpha(v)\to\beta(v)$ is assigned the label $\xi_v=y_{\alpha(v)}/y_{\beta(v)}$.

The sum $\Smin{\Ga}{\yy|\mathfrak{A}}$ becomes simply 
$\sum_{v\in\VV} \xi_v$. However, the values $\xi_v$ are inherently dependent: to every loop $v_1\to v_2\to\dots\to v_h$ in $\Ga$ there corresponds the relation $\xi_{v_1}\xi_{v_2}\cdot\dots\cdot\xi_{v_h}=1$.
The problem of minimization of sums under constraints
of this type is studied in the author's work \cite{Sadov2020}, which hereinafter will be 
%frequently 
referred to several times.
 
\begin{figure}
\label{fig:ex-arr-red}
\begin{tikzpicture}
\begin{scope}[every node/.style={circle,thick,draw}]
    \node (A) at (0,3) {$x_1$};
    \node (B) at (3,3) {$x_2$};
    \node (C) at (3,0) {$x_3$};
    \node (D) at (0,0) {$x_4$};

    \node (P1) at (5,3) {$y_2$};
    \node (P2) at (8,3) {$y_3$};
    \node (P3) at (8,0) {$y_4$};
    \node (P4) at (5,0) {$y_1$};

    \node (Q1) at (10,3) {$y_1$};
    \node (Q2) at (12.6,3) {$y_3$};
    \node (Q3) at (11.3,0) {$y_2$};
\end{scope}

\begin{scope}[>={Stealth[black]},
              every node/.style={fill=white,circle},
              every edge/.style={draw=black,very thick}]
    \path [->] (A) edge (B);
    \path [->] (B) edge (D);
    \path [->] (B) edge[bend right=15]  (C);
    \path [->] (C) edge[bend right=15]  (B);
    \path [->] (C) edge  (D); 
    \path [->] (D) edge  (A); 

    \path [->] (P1) edge node{$\xi_1$}(P2);
    \path [->] (P2) edge[bend right=15] node{$\xi_2$} (P3);
    \path [->] (P3) edge[bend right=15] node{$\xi_3$} (P2);
    \path [->] (P4) edge node{$\xi_4$} (P1);

    \path [->] (Q1) edge node{$\xi_1$}(Q2);
    \path [->] (Q2) edge[bend right=15] node{$\xi_2$} (Q3);
    \path [->] (Q3) edge[bend right=15] node{$\xi_3$} (Q2);
    \path [->] (Q3) edge node{$\xi_4$} (Q1);
\end{scope}

\node at (1.3,-1.2) {(a) Graph $\Ga$};
\node at (6.5,-1.2) {(b) $\Ga_{\{(4)>(1)>(2)>(3)\}}$};
\node at (11.3,-1.2) {(c) $\Ga_{\{(1)>(3,4)>(2)\}}$};
\end{tikzpicture}

\caption{Graph $\Ga$ of Example~\ref{ex:gr-arr-red}
and the reduced graphs corresponding to two different preferential arrangements}
\end{figure}
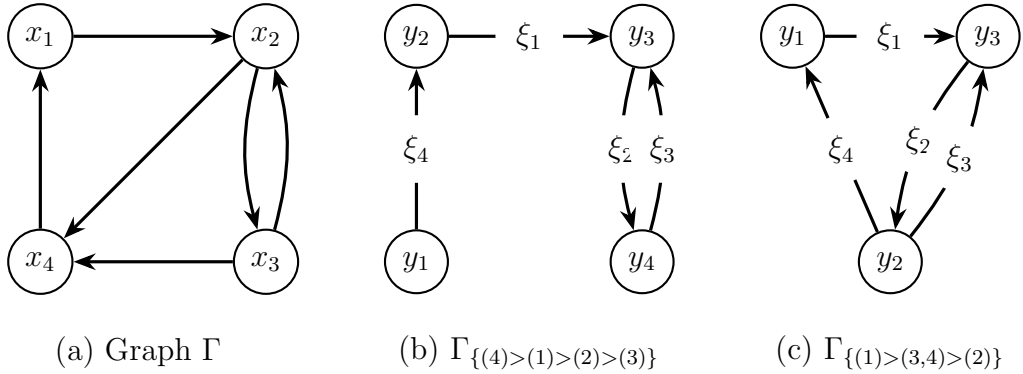

\begin{example}
\label{ex:gr-arr-red}
Consider the graph $\Ga$ with $\VV(\Ga)=[5]$ depicted
in Figure~\ref{fig:ex-arr-red}(a).
The corresponding min-sum is
$$
 \Smin{\Ga}{\xx}=\frac{x_1}{x_2}+\frac{x_2}{\land(x_3,x_4)}+\frac{x_3}{\land(x_2,x_4)}+\frac{x_4}{x_1}.
$$

Consider the preferential arrangement $\mathfrak{A}=\{(4)>(1)>(2)>(3)\}$.
If $\xx\in W_{\mathfrak{A}}$, then
$\land(x_3,x_4)=x_3$ and $\land(x_2,x_4)=x_2$.
The functions $\alpha$ and $\beta$ (see Definition~\ref{def:prefarr_red}) are:

%\bigskip
\begin{center}
\begin{tabular}{c||c|c|c|c}
$v$ & $1$ & $2$ & $3$ & $4$ \\
\hline
$\alpha(v)$ & $2$ & $3$ & $4$ & $1$ \\
\hline
$\beta(v)$ & $3$ & $4$ & $3$ & $1$ 
\end{tabular}
\end{center}

The relation between the $x$-variables and $y$-variables as defined in Lemma~\ref{lem:prefarr-red}
is: $y_1=x_4$, $y_2=x_1$, $y_3=x_2$, $y_4=x_3$. Hence
$$
 \Smin{\Ga}{\yy|\mathfrak{A}}=\frac{y_2}{y_3}+
\frac{y_3}{y_4}+\frac{y_4}{y_3}+\frac{y_1}{y_2}.
$$
In this case, the numerical labels of the arcs of the graph $\Ga_{\mathfrak{A}}$ are $\xi_1=y_2/y_3$, $\xi_2=y_3/y_4$, $\xi_3=y_4/y_3$, $\xi_4=y_1/y_2$.
The corresponding graph $\Gamma_{\mathfrak{A}}$ is shown in Fig.~\ref{fig:ex-arr-red}(b).
Note the cycle of length 2 in  $\Gamma_{\mathfrak{A}}$ and the relation $\xi_2\xi_3=1$.

Consider now the preferential arrangement $\mathfrak{A}=\{(1)>(3,4)>(2)\}$.
%If $\xx\in W_{\mathfrak{A}}$, then
%$\land(x_3,x_4)=x_3$ and $\land(x_2,x_4)=x_2$.
Here we have three blocks; the $y$-variables 
corresponding to a vector $\xx\in W_{\mathfrak{A}}$ are:
$y_1=x_1$, $y_2=x_3=x_4$, $y_3=x_2$.
The functions $\alpha$ and $\beta$ are:

%\bigskip
\begin{center}
\begin{tabular}{c||c|c|c|c}
$v$ & $1$ & $2$ & $3$ & $4$ \\
\hline
$\alpha(v)$ & $1$ & $3$ & $2$ & $2$ \\
\hline
$\beta(v)$ & $3$ & $2$ & $3$ & $1$ 
\end{tabular}
\end{center}

The relation of Lemma~\ref{lem:prefarr-red} becomes
$$
 \Smin{\Ga}{\yy|\mathfrak{A}}=\frac{y_1}{y_3}+
\frac{y_3}{y_2}+\frac{y_2}{y_3}+\frac{y_2}{y_1}.
$$
The graph $\Gamma_{\mathfrak{A}}$ is shown in Fig.~\ref{fig:ex-arr-red}(c). It has two independent cycles, $y_1\to y_3\to y_2\to y_1$ and $y_3\to y_2\to y_3$, corresponding to the relations
$\xi_1\xi_2\xi_4=\xi_2\xi_3=1$.
\end{example}

\begin{lemma}
\label{lem:minsum-minimizer}
Suppose $\Ga$ is a strongly connected graph. Then the symbol $\inf$ in the definition \eqref{infSgen} of $\mmin{\Ga}$ can be replaced by $\min$. That is, there exists a vector $\xx\in\RR_+^{\VV(\Ga)}$
such that $\Smin{\Ga}{\xx}=\mmin{\Ga}$.
\end{lemma}

\begin{proof}
Let $\mu=\mmin{\Ga}$, $\mu_0=\min(\mu,1)$, and $n=|\VV(\Ga)|$.

For any $v$ and $v'$ such that $v'\in\Ga^+(v)$ we have
$\Smin{\Ga}{\xx}\geq x_{v}/x_{v'}$, hence
$x_{v}/x_{v'}\geq \mu$ or, equivalently,
$x_{v'}/x_{v}\leq \mu^{-1}$.

Since $\Ga$ is strongly connected,
for any two nodes $v'$ and $v''$ there is a path from
$v''$ to $v'$ with at most $n-1$ arcs.
It follows that $\mu_0^{1-n}\geq x_{v''}/x_{v'}\geq\mu_0^{n-1}$.

Imposing the normalization condition $x_{v^*}=1$ for
some $v^*\in\VV(\Ga)$, we see that
$\Smin{\Ga}{\xx}>\mmin{\Ga}$ outside the cube
$K=\{\xx\mid \mu_0^{1-n}\leq x_v\leq \mu_0^{n-1}, \;\forall v\in\VV(\Ga)\}$. 
Since the function $\Smin{\Ga}{\cdot}$ is continuous on $K$, it does attain the minimum value at some $\xx\in K$.
\end{proof}

Any vector $\xx$ such that $\Smin{\Ga}{\xx}=\mmin{\Ga}$
is called a {\em minimizer} (for $\Smin{\Ga}{\cdot}$).
Clearly, if $\xx$ is a minimizer, then so is $t\xx$
for any $t>0$.
%We will sometimes need minimizers normalized in a special way: we call $\xx$ a {\em normal minimizer}\ if $\prod_{v\in\VV(\Ga)} x_v=1$. Lemma~\ref{lem:minsum-minimizer} guarantees the existence of a normal minimizer if $\Ga$ is strongly connected.

%The next theorem is crucial for our algorithm to compute $\mmin{\Ga}$.

\begin{theorem}
\label{thm:chamber_optimization}
Suppose $\Ga$ is a strongly connected graph and $\xx^*$
is a minimizer for $\Smin{\Ga}{\cdot}$. Let $\mathfrak{A}=(B_1>\dots>B_k)$
be the preferential arrangement of $\xx^*$-type. Let
$\yy^*\in\RR_+^k$ be the vector defined by $y^*_{\alpha(v)}=x^*_v$ (in the notation of Definition~\ref{def:prefarr_red}).
Then $\yy^*$ is a minimizer for the 
function $\yy\mapsto\Smin{\Gamma}{\yy|\mathfrak{A}}$.
%, unique up to a positive multiple.
\end{theorem}

\begin{proof}
According to Remark after Lemma~\ref{lem:prefarr-red},
the vector $\yy^*$ lies in the cone
$$
K=\{\yy\mid y_1>\dots>y_k>0\}.
$$
The fact that $K$ is an {\em open}\ set is crucial for our argument.

Consider two cases (cf.~Fig.~\ref{fig:ex-arr-red}(b) and (c)).

\smallskip
{\em Case 1}: The graph $\Ga_{\mathfrak{A}}$ has a node of indegree $0$. Equivalently, some variable
$y_j$ never appears as a denominator in the sum
$\Smin{\Ga}{\yy|\mathfrak{A}}$ (that is, in the right-hand side of \eqref{pa-qsum}). Let $\eps>0$ be so small that the vector $\yy^\#$ obtained from $\yy^*$
by changing $y^*_j$ into $y^*_j-\epsilon$ still lies in $K$. Then the vector $\xx^\#$ with components $x^\#_v=y^\#_{\alpha(v)}$ lies in $W_{\mathfrak{A}}$.
By Lemma~\ref{lem:prefarr-red} we have
$\Smin{\Ga}{\yy^*|\mathfrak{A}}=\Smin{\Ga}{\xx^*}$
and
$\Smin{\Ga}{\yy^\#|\mathfrak{A}}=\Smin{\Ga}{\xx^\#}$.

Clearly, the function $\Smin{\Ga}{\yy}$ is monotone (increasing) with respect to $y_j$, so
$\Smin{\Ga}{\yy^\#|\mathfrak{A}}<\Smin{\Ga}{\yy^*|\mathfrak{A}}$. 
Hence, $\Smin{\Ga}{\xx^\#}<\Smin{\Ga}{\xx^*}$, a contradiction.

\smallskip
{\em Case 2}: The graph $\Ga_{\mathfrak{A}}$ does not have nodes of indegree $0$. Equivalently, every variable $y_j$ appears as the denominator of some term in the right-hand side of \eqref{pa-qsum}. 

The optimization problem
\beq{arr-quotient}
\Smin{\Ga}{\yy|\mathfrak{A}}\to\inf,
\qquad \yy>0,
\eeq
is equivalent, by putting $\xi_v=y_{\alpha(v)}/y_{\beta(v)}$, to the problem
\beq{arr-prodcons}
\sum_{v\in\VV(\Ga)} \xi_v\to\inf,
\qquad \xi_v>0,\quad
\prod\xi_v^{\epsilon_j(v)}=1,\quad j=1,\dots,m,
\eeq
where $m$ is the number of independent cycles in the graph $\Ga_{\mathfrak{A}}$ and $\eps_j(v)=1$ if the arc
$\alpha(v)\to\beta(v)$ belongs to the cycle number $j$, otherwise $\eps_j(v)=0$. (Cf.~\cite[Sec.~I.4.1, Problem~40]{Sadov2020}.)

By an argument similar to the proof of Lemma~\ref{lem:minsum-minimizer} we see that $\inf$ in \eqref{arr-prodcons} can be replaced by $\min$ and the minimzation can be carried over a compact set. 
%Then, by the terminology of \cite[Sec.~I.1]{Sadov2020}, the problem \eqref{arr-prodcons} is {\em compact} and all variables $\xi_j$ are {\em essential}. 
%$\Smin{\Ga}{\yy|\mathfrak{A}}$. 
Moreover, by 
\cite[Problem~15]{Sadov2020}, the minimizer $\vec{\xi}^{**}$ for the problem \eqref{arr-prodcons} is determined as the unique critical point (the solution of a system of plolynomial equations obtained through Lagrange's multipliers method). 

Let ${\vec\xi}^*$ be the vector corresponding to $\yy^*$, i.e.\ $\xi^*_v=y^*_{\alpha(v)}/y^*_{\beta(v)}$ for
all $v\in\VV(\Ga)$.
Take some neighborhood $\mathcal{O}\subset K$ of $\yy^*$ (it exists, since $K$ is open). It is easy to see that the map $\yy\mapsto\vec{\xi}$ is open (in fact, its restriction on some ``normalizing'' hypersurface, say, $\sum_{j=1}^k y_j=1$ is a homeomorphism), so the image $\Omega$ of $\mathcal{O}$ is an open neighborhood of ${\vec\xi}^*$.

Suppose that ${\vec\xi}^{**}\neq{\vec\xi}^{*}$. 
Then, since ${\vec\xi}^{*}$ is not a critical point (which is unique), there exists some ${\vec\xi}^\#\in\Omega$ such that $\sum_v\xi_v^\#<\sum_v\xi_v^*$. Hence there exists
$\yy^\#\in\mathcal{O}$ such that $\Smin{\Ga}{\yy^\#|\mathfrak{A}}<\Smin{\Ga}{\yy^*|\mathfrak{A}}$. Now, as in Case 2, we find the corresponding
vector $\xx^\#\in W_{\mathfrak{A}}$ such that 
$\Smin{\Ga}{\xx^\#}<\Smin{\Ga}{\xx^*}$ and conclude that $\xx^*$ is not a minimizer, a contradiction.

\smallskip
The conclusion is: Case 2 takes place and $\vec{\xi}^{**}=\vec{\xi}^*$; so $\yy^*$ is a minimizer for the optimization problem \eqref{arr-quotient}.
Q.E.D.
\end{proof}

\begin{remark}
A minimizer for $\Smin{\Ga}{\cdot}$ is not necessarily unique up to a scalar multiple.
Consider, for example,
$$
\Smin{\Ga}{\xx}=\frac{x_1}{\land(x_1,x_2)}+\frac{x_2}{x_3}+\frac{x_3}{x_2}.
$$
Clearly, $\Smin{\Ga}{\xx}\geq 3$ for any $\xx$. The set of $\xx^*$ such that $\Smin{\Ga}{\xx}=3$ is two-parametric: $x^*_1=a, x^*_2=x^*_3=b$, where $0<a\leq b$.

This example also shows that the graph $\Ga_{\mathfrak{A}}$ corresponding to the preferential arrangment $\mathfrak{A}\ni\xx^*$ is not necessarily (weakly) connected.
\end{remark}

Based on Theorem~\ref{thm:chamber_optimization}, we
formulate a theoretical algorithm for computation of $\mmin{\Ga}$ for the given graph $\Ga$.

\begin{enumerate}
\item 
Make the list $\Arr_\Ga$ of all preferential arrangments  of the set ${\VV(\Ga)}$.

\item 
Select the subset $\Arr'_\Ga$ comprising those preferential arrangments $\mathfrak{A}$ for which the graph $\Ga_{\mathfrak{A}}$ does not have a node of indegree zero. (Cf.~Case~1 of proof of Theorem~~\ref{thm:chamber_optimization}.)

\item 
For every $\mathfrak{A}\in \Arr'_\Ga$ find a minimizer $\yy^{(\mathfrak{A})}$ (unique up to a positive multiple) for the problem \eqref{arr-quotient}
by solving the corresponding system of polynomial equations. (The relevant material in \cite{Sadov2020} is: Problem~40, Eqs.~(11), (7), and Problem~21.)
Put $\mmin{\Ga,\mathfrak{A}}=\Smin{\Ga}{\yy^{(\mathfrak{A})}\mid \mathfrak{A}}$. 

\item 
Select the subset $\Arr''_\Ga\subset\Arr'_\Ga$ comprising those preferential arrangments $\mathfrak{A}=(B_1>\dots>B_k)$ ($k$ depends on $\mathfrak{A}$) for which the components $y^{(\mathfrak{A})}_j$, $j=1,\dots,k$, form a strictly decreasing sequence.
%(preferrential arrangements consistent with the minimizer)

\item 
The answer is given by the formula
\beq{smin-theor-answer}
\mmin{\Ga}
=\min_{\mathfrak{A}\in\Arr''_\Ga}
\mmin{\Ga,\mathfrak{A}}.
\eeq
\end{enumerate}

\begin{remark}
We have shown that the problem of finding $\mmin{\Ga}$ can be reduced to solution of finitely many systems of nonlinear algebraic equations.
However the above algorithm has little value as a practical method because of its tremendous combinatorial complexity: the number of preferential arrangements on a $n$-element set is asymptotic to
$\frac{1}{2}(\log 2)^{-n-1} n!$ \cite[Eq.~(3.4.26)]{Bergeron-Labelle-Leroux_1998}.
It would be interesting to devise 
%The author is not aware of 
a method of lower complexity.
\end{remark}

%Notwithstanding this remark, Theorem~\ref{thm:chamber_optimization} is not a dead-end result. We close this section with a demonstration of its theoretical application.

%(Applications - in Extr.Prob.section?)

\section{Min-sums II: Extremal problems}
\label{sec:min-extr}

The optimization problem:
find $\mmin{\Ga}$ for the given $\Ga$ --- will be solved here only in some special cases. The question that turned out more amenable and will be explored rather satisfactorily is an extremal graph problem undestood in the standard sense as ``determining the extreme value of some graph parameter over some class of graphs''
\cite[p.~373]{West_2001}. The parameter in present case is $\mmin{\Ga}$
and the class of (directed) graphs consists of graphs of positive minimum outdegree with given number of nodes (and possibly additional constraints).  

\iffalse
\begin{quote}
Given a property $Q$ and an invariant $\mu$ for a class $\mathcal{H}$ of set systems, what is the maximum (minimum) of $\mu$
over all set systems in $\mathcal{H}$ satisfying $Q$?
\end{quote}
B. Bollobas, Combinatorics, set systems, hypergraphs
--- p.53 (Extremal problems)
\fi

We denote by $\GraphsNN{n}$ the class of (isomorphism classes of) directed graphs with $n$ nodes, positive minimum outdegree, positive minimum indegree (so that every variable $x_v$ in the corresponding sum appears in the denominator of at least one term) and no multiple arcs.
\smallskip

\iffalse
We will need some combinatorial terminology.

With the graph $\Ga$ one can associate the so called {\em transversal matroid} \cite{MatroidText}.
We will not need the matroid proper, only the notion of its rank.

The rank $r_\Ga$ of the transversal matroid associated with graph $\Ga$ is the maximum number $r$ such that
there exists a partial system of distinct representatives of cardinality $r$ for the family $\{\Ga^+(v)\}$, that is,
two sets of cardinality $r$, $\{u_1,\dots,u_r\}\subset\VV(\Ga)$ and $\{v_1,\dots,v_r\}\subset\VV(\Ga)$, such that 
$u_i\in\Ga^+(v_i)$.

The rank has the greatest possible value $r_\Ga=|\VV(\Ga)|$ if and only if the
system $\{\Ga^+(v)\}_{v\in\VV(\Ga)}$ admits an SDR.
\fi

\iffalse
%In "Prodconstraints":
Let $\mathfrak{G}_m$ denote the set of of all (isomorphism classes of) strongly connected, not necessarily simple (di)graphs with
$m$ arcs and let $\mathfrak{G}_{m,n}$ ($n\leq m$) denote the subset of $\mathfrak{G}_m$ consisting of graphs with $n$ nodes.
\fi

We begin with the question about the maximum possible value
of $\mmin{\Ga}$ for $\Ga\in\GraphsNN{n}$.
If $\Ga$ is a graph such that set system $\{\Ga^+(v)\}$ admits an SDR, or, equivalently, there exists a $\Ga$-admissible bijection of $\VV(\Ga)$,
we will say for short that {\em $\Ga$ admits and SDR}.

Let us revisit Proposition~\ref{prop:lbnd-minsum-sigma}
and consider its simplest case where $\Ga$ admits an SDR.
Then the upper bound \eqref{lbnd-minsum-sigma} for $\mmin{\Ga}$ coinsides with lower bound in \eqref{gr-psum-lbnd}, so $\mmin{\Ga}=n$.
%The natural question arises, whether 
But does 
the latter equality 
%occurs 
take place
{\em only if}\ there exists an SDR?
The answer in affirmative is given below.

\begin{theorem} 
\label{thm:min-sdr}
(a) Let $\Ga\in\GraphsNN{n}$. The upper bound $\mmin{\Ga}=n$ is attained if and only if $\Ga$ admits an SDR. 

(b) The maximum value excluding SDR-admitting graphs is 
\beq{maxminsum_no_SDR}
 \max\left(\mmin{\Ga}\mid \Ga\in\GraphsNN{n}, \;\text{\rm $\Ga$ does not admit an SDR}\right) = n-\Delta_n,
\eeq
where $\Delta_{2k-1}=\Delta_{2k}=2k-1-2\sqrt{k(k-1)}$. %If the class $\GraphsN{n}$ in the left-hand side of \eqref{maxminsum_no_SDR} is reduced to $\GraphsNN{n}$, then $\Delta(n)$ in the right-hand side is to be replaced by $\Delta_{n-1}$.
(Note: $\Delta_n\sim (2n)^{-1}$ as $n\to\infty$.)
\end{theorem}

\begin{proof}
(a) We have to prove the ``only if'' part.
Let $\VV=\VV(\Ga)$ and $n=|\VV|$. Suppose, contrary to what the theorem claims, that $\mmin{\Ga}=n$, yet $\Ga$ does not admit an SDR.
Then by Hall's theorem there exist sets $\Omega\subset \VV$ and $\Phi\subset \VV$
such that $\Phi=\cup_{v\in\Omega}\Ga^+(v)$ and
$|\Phi|<|\Omega|$. Clearly,
$$
 \Smin{\Ga}{\xx}\leq \sum_{v\in\Omega}\frac{ x_v}{\min\limits_{v\in\Phi}x_v}+\sum_{v\in\VV\setminus\Omega}\frac{ x_v}{\min\limits_{v\in\VV}x_v}.
$$
Let us take the vector $\xx$ with components
$$
 x_v=\begin{cases}
1,\quad\text{if $v\in\Phi$};\\
t,\;\;\text{if $v\in\VV\setminus\Phi$},
\end{cases}
$$
where $0<t\leq 1$.
Then
$$
\Smin{\Ga}{\xx}\leq |\Omega\cap\Phi|+t|\Omega\setminus\Phi|
+\frac{|\Phi\setminus\Omega|}{t}
+|\VV\setminus(\Omega\cup\Phi)|.
$$
Put $a=|\Omega\setminus\Phi|$, $b=|\Phi\setminus\Omega|$. We have $a-b=|\Omega|-|\Phi|>0$.
Also,
$$
\min_{0<t\leq 1}\left(ta+\frac{b}{t}\right)=2\sqrt{ab}.
$$
Put $\delta(a,b)=a+b-2\sqrt{a+b}$. Clearly, $\delta(a,b)>0$. Hence
$$
 \Smin{\Ga}{\xx}\leq |\Omega\cap\Phi|+|\Omega\setminus\Phi|
+|\Phi\setminus\Omega|
+|\VV\setminus(\Omega\cup\Phi)|-\delta(a,b)<n.
$$
The assumption that $\Ga$ does not admit an SDR has led to a contradiction.

\smallskip
(b) To prove \eqref{maxminsum_no_SDR}, note first that the function $\delta(a,b)$ decreases in $b$ when $a$ is fixed. Since $a+b\leq n$ and $b\leq a-1$, we have
$\delta(a,b)\geq\delta(a,\min(n-a,a-1))$. 

The function $a\mapsto\delta(a,a-1)$ decreases for $a\geq 1$, while the function $a\mapsto\delta(a,n-a)=n-\sqrt{a(n-a)}$ increases for $n/2\leq a\leq n$. It follows that for any fixed $n\geq 2$
and $k=\left\lfloor(n+1)/2\right\rfloor$
$$
 \min\delta(a,b)=\delta(k,k-1)=\Delta_{n}.
$$

The obtained lower bounds for $\delta(a,b)$ imply corresponding upper bounds for $\mmin{\Ga}$.
It remains to show that those upper bounds are attainable.

Let $\VV(\Ga)=[1,n]$, $\Omega=[1,k]$, $\Phi=[n-k+2,n]$
(note: $n-k+2>k$). Define $\Ga\in\GraphsNN{n}$ by
$\Ga^+(i)=\Phi$ for $i\in[1,k]$ and $\Ga^+(i)=[1,n]$
for $i\in[k+1,n]$. Denote $\land_{i\in[n]} v_i=p$,
$\land_{i\in\Phi}=q$. 
Clearly, $p=\min(q,x_{1},\dots,x_{n-k+1})$. Hence
\begin{multline*}
 \Smin{\Ga}{\xx}=\frac{x_1+\dots+x_k}{q}+\frac{x_{k+1}+\dots+x_{n}}{p}\geq
\frac{kp}{q}+\frac{(k-1)q}{p}+(n-2k+1)
\\
\geq 2\sqrt{k(k-1)}+(n-2k+1)=n-\delta(k,k-1).
\end{multline*}
Therefore $\mmin{\Ga}\geq n-\delta(k,k-1)$. In combination with previously proved upper bound it shows that $\mmin{\Ga}=n-\delta(k,k-1)$.
\end{proof}

Next comes the question about the minimum possible value of $\mmin{\Ga}$ for $\Ga\in\GraphsNN{n}$. Taken at face value, it is quite trivial.

\begin{proposition}
For any $n\geq 2$
$$
 \min_{\Ga\in\GraphsNN{n}}\mmin{\Ga}=2.
$$
%For $n>2$ the lower bound is not attained.
\end{proposition}

\noindent
\begin{minipage}{0.8\textwidth}
\begin{proof}
Suppose first that there is a strong component in $\Ga$ with at least 2 nodes.
Then it contains a cycle of length $\geq 2$.
Hence, by Proposition~\ref{prop:strongred-min},
$\mmin{\Ga}\geq 2$.

\smallskip
Now, if all strong components of $\Ga$ are singletons,
then at least two of them (one final and another --- either also final or one without predecessor in the condensation) contain a loop. Again, it follows that  
$\mmin{\Ga}\geq 2$.

\smallskip
The graph depicted on the right
%in Figure~\ref{fig:minminsum} 
provides the lower bound $\mmin{\Ga}=2$, by Proposition~\ref{prop:strongred-min}.
\end{proof}
\end{minipage}
\hspace{0.08\textwidth}
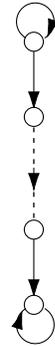
\begin{minipage}{0.08\textwidth}
\begin{tikzpicture}[x=0.5cm,y=0.5cm]
%--Lin(4)=Palm(4,1)---------------------------------------
\draw [-{Latex[scale=1.5,width=3pt,color=black]}] (2,2)--(2,0.17);
\draw [-{Latex[scale=1.5,width=3pt,color=black]}] (2,7)--(2,5.17);
\draw [line width=0.5pt , dash pattern= on 1.8pt off 2.2pt] (2,5)--(2,3.7);
\draw [-{Latex[scale=1.5,width=3pt]}] (2,3.65)--(2,3);
\draw [line width=0.5pt , dash pattern= on 1.8pt off 2.2pt] (2,3)--(2,2.05);
\draw [fill=white] (2,0) circle (3.6pt);
\draw [fill=white] (2,2) circle (3.6pt);
\draw [fill=white] (2,5) circle (3.6pt);
\draw [fill=white] (2,7) circle (3.6pt);
\node at (2,-0.1) (Final) {};
\draw [-] (Final.east) arc(60:-240:0.5);
\node at (2,7.1) (Initial) {};
\draw [-] (Initial.east) arc(-60:240:0.5);
\draw [-{Latex[scale=1.5,width=3pt,color=black]}] (1.55,-0.63)--(1.7,-0.14);
\draw [-{Latex[scale=1.5,width=3pt,color=black]}] (2.45,7.37)--(2.35,7.1);
\end{tikzpicture}
\end{minipage}

\bigskip
The question becomes more interesting and challenging if
we restrict the class $\GraphsNN{n}$ to its subclass
$\GraphsN{n}$ of strongly connected graphs. 

Problem~74(b,c) in \cite{Sadov2020} asks, in different terms, to prove that for $n=2021$
$$
 20<\min_{\Ga\in\GraphsN{n}}\mmin{\Ga}<21.
$$

The general result is as follows.%
\footnote{The combinatorial part of solution of Problem~74  
given in \cite{Sadov2020} can be converted to that of a general proof by %mechanical 
changing $2021$ into $n$ and making obvious induced replacements. 
We give an independent proof here, in part because the proof
of one of the key lemmas (Problem~71) in \cite{Sadov2020} ver.~1  is grossly inaccurate.}

\begin{theorem}
\label{thm:minminsum-strong}
For any $\Ga\in\GraphsN{n}$
\beq{minminsum-strong}
 \mmin{\Ga}>e\,\ln(n+1-\ln(n+1)).
\eeq
Furthermore,
\beq{asymminsum-strong}
 \min_{\Ga\in\GraphsN{n}}\mmin{\Ga}=e\ln n+O\left(\frac{1}{\ln n}\right)
\eeq
as $n\to\infty$.
\end{theorem}

\begin{proof} 
Suppose $\xx$ is a minimizer for $\Smin{\Ga}{\cdot}$.
Put $x_*=\min(x_v, v\in\VV)$, $x^*=\max(x_v,v\in\VV)$.
Let $v_*\neq v^*$ be two nodes such that
$x_{v_*}=x_*$ and $x_{v^*}=x^*$.

Consider a simple (not self-intersecting) path $\pi$ from $v^*$ to $v_*$. Let us label the nodes on the path $\pi$ by numbers from $1$ to $k+1$ so that 
$\pi=(v_1\to v_1\dots\to v_{k+1})$. In particular,
$v_1=v^*$, $v_{k+1}=v_*$.
If $k+1<n$, we label the remaining nodes by numbers $k+2,\dots,n$ arbitrarily.
The node values $x_{v_k}$ will be written as $x_k$ for short.

Obviously, $\land(\xx|\Ga^+(v_i))\leq x_{i+1}$
for $i=1,\dots,k$.
Hence
$$
 \sum_{i=1}^{k}\frac{x_i}{\land(\xx|\Ga^+(v_i))}
\geq \sum_{i=1}^{k}\frac{x_i}{x_{i+1}}\geq
k\left(\frac{x^*}{x_*}\right)^{1/k}.
$$
(The last step uses the AGM inequality.)

For the remaining nodes we use the trivial estimates
 $\land(x_i|\Ga^+(i))\leq x^*$
and
$x_i\geq x_*$. 
Therefore
$$
 \sum_{i=k+1}^{n}\frac{x_i}{\land(\xx|\Ga^+(v_i))}
\geq (n-k)\frac{x_*}{x^*}.
$$

We conclude that
$$
 \mmin{\Ga}=\Smin{\Ga}{\xx}\geq \min_{1\leq k\leq n-1}\,\min_{t>0} \left(kt^{1/k}+\frac{n-k}{t}\right).
$$
The minimum of the inner function equals $(k+1)(n-k)^{\frac{1}{k+1}}$, attained at $t=(n-k)^{\frac{k}{k+1}}$.

Let $z=k+1$ and $r=n+1$. Then 
$$
\ln\left((k+1)(n-k)^{\frac{k}{k+1}}\right)=f(z,r)\,\stackrel{\text{def}}{=}\,
\ln z+\frac{\ln(r-z)}{z}.
$$
It can be shown that $\min_{1<z\leq r-1} f(z,r)
>1+\ln\ln(r-\ln r)$.
The calculation is straightforward but not too short; details can be found in \cite[Solution of Problem~67B]{Sadov2020} (an even more precise estiimate for $\min_z f(z,r)$ is given in \cite{KalachevSadov_2017}). 
The inequality \eqref{minminsum-strong} follows.

From our analysis one sees that the equality
$$
\Smin{\Ga}{\xx}=k\left(\frac{x^1}{x_{k+1}}\right)^{1/k}+
(n-k)\frac{x_{k+1}}{x^1}
$$
takes place (at least) for the graph $\Ga_k$ with
$\VV(\Ga_k)=[n]$,
$$
\ba{l}
\Ga_k^+(i)=\{i+1\},\quad i=1,\dots,k-1;
\\
\Ga_k^+(k)=\{k+1,\dots,n\};
\\
\Ga_k^+(i)=\{1\},\quad i=k+1,\dots,n,
\ea
$$
provided that $x_1\geq x_2\geq\dots\geq x_{k+1}=\dots=x_n$. 

The corresponding min-sum is
$$
\Smin{\Ga_k}{\xx}=\frac{x_1}{x_2}+\dots+\frac{x_{k-1}}{x_k}+\frac{x_k}{\land(x_{k+1},\dots,x_n)}+
\frac{1}{x_1}\sum_{i=k+1}^n x_i.
$$

One can find the asymptotics ($r\to\infty$) of the minimum of the function $f(z,r)$ (again, we refer to \cite{Sadov2020} or \cite{KalachevSadov_2017} for details) and obtain
$$
\min_{1<z\leq r-1} \exp f(z,r)=e\ln r+O\left(\frac{1}{\ln r}\right).
$$
The asymptotic formula remains true if $z$ is allowed to assume only integer values. Putting $k=z$ and minimizing over $k$, we find that in the extremal case
$\mmin{\Ga_k}=e\ln(n+1)+O((\ln n)^{-1})$, which is equivalent to \eqref{asymminsum-strong}.
\end{proof}

\end{document}